\documentclass[reqno,11pt]{amsart}
\usepackage{latexsym}
\usepackage{euscript}
\usepackage{amssymb}
\usepackage{float}      
\usepackage{array}      
\usepackage{graphicx}
\usepackage{epic}            
\usepackage{enumerate}   
\usepackage{color}

\newtheorem{Thm}[equation]{Theorem}
\newtheorem{Lem}[equation]{Lemma}
\newtheorem{Cor}[equation]{Corollary}
\newtheorem{Prop}[equation]{Proposition}

\theoremstyle{remark}

\newtheorem{Rem}[equation]{Remark}
\newtheorem{Example}[equation]{Example}
\newtheorem{Def}[equation]{Definition}

\numberwithin{equation}{section}

\newcommand{\R}{\mathbb R}           
\newcommand{\C}{\mathbb C}           
\newcommand{\Z}{\mathbb Z}           

\newcommand{\Ad}{\rm Ad}
\newcommand{\ad}{\rm ad}
\newcommand{\rank}{\rm rank}

\newcommand{\Ker}{\rm Ker}
\newcommand{\Hom}{\rm Hom}

\newcommand{\GR}{\underline{\text{Gr}}}

\newcommand{\f}[1]{{\mathfrak{#1}}}           
\newcommand{\fb}{{\mathfrak b}}
\newcommand{\fg}{{\mathfrak g}}
\newcommand{\fh}{{\mathfrak h}}
\newcommand{\fk}{{\mathfrak k}}
\newcommand{\fl}{{\mathfrak l}}

\newcommand{\fn}{{\mathfrak n}}

\newcommand{\fp}{{\mathfrak p}}
\newcommand{\fq}{{\mathfrak q}}

\newcommand{\fu}{{\mathfrak u}}

\newcommand{\fB}{{\mathfrak B}}

\newcommand{\ga}{\alpha}
\newcommand{\gb}{\beta}

\newcommand{\gl}{\lambda}
\newcommand{\gL}{\Lambda}
\newcommand{\gd}{\delta}
\newcommand{\gD}{\Delta}

\renewcommand{\gg}{\gamma}
\newcommand{\gG}{\Gamma}
\newcommand{\gs}{\sigma}

\newcommand{\gt}{\theta}

\newcommand{\eps}{\varepsilon}

\newcommand{\Cal}{\mathcal}

\newcommand{\cN}{{\mathcal N}}
\newcommand{\cO}{{\mathcal O}}
\newcommand{\cE}{{\mathcal E}}
\newcommand{\cR}{{\mathcal R}}
\newcommand{\cS}{{\mathcal S}}
\newcommand{\cM}{\mathcal M}

\newcommand{\cF}{\mathcal F}
\newcommand{\cB}{\mathcal B}
\newcommand{\cU}{\mathcal U}

\newcommand{\bK}{\mathbf K}

\renewcommand{\bar}[1]{\overline{#1}}

\newcommand{\IP}[2]{\langle#1\,, #2\rangle}     

\newcommand{\supp}{\rm supp}
\newcommand{\ann}{\rm ann}
\newcommand{\Spin}{\rm Spin}
\newcommand{\Sp}{\rm Sp}
\newcommand{\even}{\rm even}
\newcommand{\odd}{\rm odd}
\newcommand{\Lie}{\rm Lie}

\usepackage{mathdots}  
\newcommand{\Ind}{{\rm Ind}}
\newcommand{\dual}{\,\widehat{\;}\,}
\newcommand{\X}{X_\fb}
\newcommand{\XL}{X_{\fb\cap\fl}}
\newcommand \unb[2]{\underset{#1}{{\underbrace{#2}}}}
\newcommand{\gr}{{\rm gr}}
\newcommand{\Rep}{{\rm Rep}}

\begin{document}

\parskip=4pt 
\baselineskip=14pt

\title[]{Dirac Index and associated cycles of Harish-Chandra modules}
\author{Salah~Mehdi}
\address{Institut Elie Cartan de Lorraine\\
CNRS - UMR 7502  \\
Universit\'e de Lorraine, Metz, F-57045, France}
\email{salah.mehdi@univ-lorraine.fr}

\author{Pavle~Pand\v{z}i\'{c}}
\address{  Department of Mathematics \\
Faculty of Science \\
  University of Zagreb \\
  Zagreb, Croatia}
\email{pandzic@math.hr}
 
\author{David A.~Vogan}
\address{  Department of Mathematics \\
 Massachusetts Institute of Technology \\
 Cambridge, MA 02139, USA}
\email{dav@math.mit.edu}

\author{Roger Zierau}
\address{  Department of Mathematics \\
  Oklahoma State University \\
  Stillwater, Oklahoma 74078, USA}
\email{roger.zierau@okstate.edu}
 


\thanks{The second named author was supported by grant no. 
4176 of the Croatian Science Foundation and by the QuantiXLie Centre of Excellence, a project
cofinanced by the Croatian Government and European Union through the European Regional Development Fund - the Competitiveness and Cohesion Operational Programme 
(KK.01.1.1.01.0004). The third named author was supported 
in part by NSF grant DMS 0967272.}

\begin{abstract}
Let $G_{\R}$ be a simple real linear Lie group with maximal compact subgroup $K_{\R}$ and assume that $\rank(G_\R)=\rank(K_\R)$.     For  any representation $X$ of Gelfand-Kirillov dimension  $\frac12 \dim(G_{\R}/K_{\R})$, we consider the polynomial on the dual of a compact Cartan subalgebra given by the dimension of the Dirac index of members of the coherent family containing  $X$.   Under a technical condition involving the Springer correspondence, we establish an explicit relationship between this polynomial and the multiplicities of the irreducible components occurring in the associated cycle of $X$. This  relationship  was conjectured in \cite{MehdiPandzicVogan15}.
\end{abstract}

\maketitle



\section*{Introduction} 

Suppose $G_\R$ is a simple real  linear Lie group and $K_\R$ is a maximal compact subgroup of $G_\R$.  Much may be learned about an irreducible representation of $G_\R$ by understanding its restriction to $K_\R$.  Two invariants that depend only on this $K_\R$ structure are the associated cycle (introduced in \cite{Vogan91}) and the Dirac index (defined in \cite{MehdiPandzicVogan15}, when $\rank(G_\R)=\rank(K_\R)$).  These invariants give $W$-harmonic polynomials on the dual of a Cartan subalgebra.  They are both related to the global character theory of representations (for example, \cite{SchmidVilonen00} and \cite[Thm.~6.2]{MehdiPandzicVogan15}).  The two invariants are typically not the same; for a given representation the harmonic polynomials associated to the two invariants often have different degrees.  However, there is a family of representations, those having certain annihilators, for which the degrees do coincide and a relationship between the two invariants is conjectured in \cite[Conjecture 7.2]{MehdiPandzicVogan15}. The main point of this article is to prove this conjecture, which is stated as Theorem A below.

In order to state the theorem we need some notation. Let $\fg$ be the complexification of $\Lie(G_\R)$ and $K$ the complexification of $K_\R$.  Then there is a complex group $G$ (with Lie algebra $\fg$) containing $K$ as the fixed point group of an involution.  The Springer correspondence associates to each nilpotent $G$-orbit in $\fg^*$ an irreducible representation of the Weyl group $W$. Of interest to us is the nilpotent orbit $\cO^\C$ corresponding to the (irreducible) representation of $W$ generated by the `Weyl dimension polynomial' 
$$P_K(\gl)=\prod_{\ga\in\gD_c^+}\frac{\IP{\gl}{\ga}}{\IP{\rho_c}{\ga}}.
$$
The family of Harish-Chandra modules occurring in the conjecture are those having associated variety of annihilator equal to $\bar{\cO^\C}$.  Such Harish-Chandra modules $X$ have associated cycle of the form 
$$  AC(X)=\sum_{i=1}^m m_i(X)\bar{\cO}_i,
$$
where $\cO_1,\dots,\cO_m$ are the `real forms' of $\cO^\C$ in the sense that each $\cO_i$ is a nilpotent $K$-orbit in $(\fg/\fk)^*\subset \fg^*$ so that $G\cdot \cO_i=\cO^\C$  (\cite{Vogan91}).  The `multiplicities' $m_i(X)$ are nonnegative integers.  Placing $X$ in a coherent family allows us to view $m_i(X)$ as a harmonic polynomial on the dual of a Cartan subalgebra.  Now write $DI_p(X)$ for the harmonic polynomial associated to the Dirac index.  Our main result is the following theorem (\cite[Conjecture 7.2]{MehdiPandzicVogan15}).

\smallskip
\noindent{\bf Theorem A.} \emph{
Assume that $\rank(G)=\rank(K)$.  Let $\cO^\C$ and $\cO_1,\dots,\cO_m$ be as above, then there exist integers $c_i$ so that
\begin{equation*}
DI_p(X)=\sum_i c_im_i(X),
\end{equation*}
 for any Harish-Chandra module with $AV(\ann(X))= \bar{\cO^\C}$.
}
\smallskip

The proof of the theorem has several ingredients.  One is an extension of the definition of the associated cycle of Harish-Chandra modules to \emph{virtual} Harish-Chandra modules.  This is done in terms of the $K$-equivariant $\textup{K}$-theory of $\cN_\gt:=\cN\cap (\fg/\fk)^*$, where $\cN$ is the nilpotent cone in $\fg^*$.   Recall that the $\textup{K}$-group $\textup{K}^K(\cN_\gt)$  is the Grothendieck group of the abelian category of $K$-equivariant coherent sheaves on $\cN_\gt$.   Although not expressed in terms of the $K$-equivariant $\textup{K}$-theory, a definition of associated cycle of virtual Harish-Chandra modules is discussed in \cite{Vogan98}.  Much of the relevant sheaf theory appears in \cite{Achar01}.  The  map $\gr$ (with respect to a `good' filtration) is a map from Harish-Chandra modules to finitely generated $(S(\fg),K)$-modules supported in $\cN_\gt$.  Since, for affine varieties, $K$-equivariant coherent sheaves are in 
one-to-one correspondence with finitely generated $(S(\fg),K)$-modules, we may consider
$$ \gr:\{\text{virtual Harish-Chandra modules}\} \to \textup{K}^K(\cN_\gt).$$
The important fact is that we may specify a $\Z$-basis of $\textup{K}^K(\cN_\gt)$ in terms of homogeneous bundles on the (finitely many) $K$-orbits in $\cN_\gt$ and use this basis to express the associated cycle.  Write $\cO=K\cdot e=K/K^e$ and let $\tau\in (K^e)\dual$, where $(K^e)\dual$ is the set of irreducible algebraic representations of $K^e$.  Then the sheaf of sections of the homogeneous vector bundle  $K\underset{K^e}{\times}\tau\to\cO$ extends 
(although not uniquely)
to a  $K$-equivariant coherent sheaf on $\bar{\cO}$, then extends by zero to a $K$-equivariant coherent sheaf  $\tilde{\cE}_\cO(\tau)$ on $\cN_\gt$.  We have the following theorem.

\smallskip
\noindent{\bf Theorem B.} \emph{
$\textup{K}^K(\cN_\gt)$ has $\Z$-basis $\{[\tilde{\cE}_\cO(\tau)]\}$ with $\cO$ running over all $K$-orbits in $\cN_\gt$ and $\tau\in(K^e)\dual$.  For any virtual Harish-Chandra module $X$, writing $\gr(X)=\sum_{\cO}\sum_\tau  n_{\cO,\tau}[\tilde{\cE}_\cO(\tau)]$, we have
$$AC(X)=\sum_{\cO\text{\upshape max'l } }\left( \sum_\tau n_{\cO,\tau}\dim(\tau)\right)\,\bar{\cO}_.
$$
}
\smallskip

\noindent An important point is that these ``leading coefficients'' $n_{\cO,\tau}$ are independent of the choices of extensions $\tilde{\cE}_{\cO'}(\tau')$.

A second ingredient of the proof is a formula for extensions $\tilde{\cE}_\cO(\tau)$.  This is accomplished  using a resolution of singularities of $\bar{\cO}$.    One obtains extensions in terms of virtual Harish-Chandra modules of discrete series representations (Prop.~\ref{prop:ext}).  The conjecture is then proved using a simple formula for the Dirac index of a discrete series representation (Prop.~\ref{prop:di-ds}).

In Sections \ref{sec:const-classical} and \ref{sec:const-exceptional} we list  (1) the groups $G_\R$ for which the hypothesis of Theorem A holds and (2) the complex $G$-orbits $\cO^\C$ corresponding to $P_K(\gl)$ under the Springer correspondence.             
As an example we include the computation of the constants appearing in Theorem A in one classical case; the other cases will appear in \cite{MehdiPandzicVoganZierau17b}.  A table with all constants is included for the exceptional groups; the computations were done using a computer.


\noindent {\bf Notation.}  We are concerned with Harish-Chandra $(\fg,K)$-modules.  These form an abelian category that we denote by $\cM(\fg,K)$.  The corresponding Grothendieck group is denoted by $\bK(\fg,K)$; it is the quotient of the free abelian group on the Harish-Chandra modules by the subgroup generated by all $X-Y+Z$ when there is an exact sequence $0\to X\to Y\to Z\to 0$ in $\cM(\fg,K)$.  When $X\in\cM(\fg,K)$ the corresponding element of $\bK(\fg,K)$ is often denoted by $[X]$.  Similarly, $\cM(S(\fg),K)$ is the category of finitely generated $(S(\fg),K)$-modules, and $\bK(S(\fg),K)$ denotes its Grothendieck group.  For an algebraic group $H$ (over $\C$)  we denote by $\Rep(H)$ the abelian group of virtual $H$-representations, i.e., the Grothendieck group of the category of algebraic representations of $H$. We refer to elements of $\bK(\fg,K)$ (resp., $\Rep(H)$) as virtual Harish-Chandra modules (resp., representations); they are integer combinations of (classes of) irreducibles.  

The nilpotent cone in $\fg^*$ is denoted by $\cN$ and is often identified with the nilpotent cone in $\fg$.  Our group $K$ is the fixed point group of an involution of $G$.  This determines a (complexified) Cartan decomposition $\fg=\fk+\fp$.  We write $\cN_\gt$ for $\cN\cap (\fg/\fk)^*$, which may be identified with the cone of nilpotent elements in $\fp$.

The conjecture is for the cases where  $G$ and $K$ have the same rank; under this assumption we choose a Cartan subalgebra $\fh$ of $\fg$ that is contained in $\fk$ and consider the set $\gD=\gD(\fh,\fg)$ of $\fh$-roots in $\fg$.  The set of $\fh$-roots in $\fk$ (resp., in $\fp$) is denoted by $\gD_c$ (resp., $\gD_n$).  We write $W$ (resp., $W_K$) for the Weyl group of $\gD$ (resp., $\gD_c$).    Half the sum of the roots in some positive root system $\gD^+$  (resp., $\gD_c^+, \gD_n^+$) is denoted  by $\rho$ (resp., $\rho_c,\rho_n$).  When $A$ is some set of roots, we write $\rho(A)$ for half the sum of the roots in $A$.


\section{Associated cycles and equivariant $K$-theory}\label{sec:AC}
Associated varieties and associated cycles of  Harish-Chandra modules were defined in \cite{Vogan91}.   In this section we recall these definitions and  describe how to extend the definitions to the Grothendieck group of Harish-Chandra modules.    Such an extension of the definitions to $\bK(\fg,K)$ is discussed in \cite{Vogan98}; details are given in this section and in the appendix.

First recall the definitions for Harish-Chandra modules. Let $X$ be a Harish-Chandra module.  Filter the enveloping algebra $\Cal U(\fg)$ by degree, so $\gr(\Cal U(\fg))$ is isomorphic to the symmetric algebra $S(\fg)$ which may be identified with the algebra $P(\fg^*)$ of polynomials on $\fg^*$.  Then there are `good filtrations' $\{X_n\}$ (\cite[equation (2.1)]{Vogan91}) of $X$ and for such a filtration, $\gr(X)$ is a finitely generated $S(\fg)$-module.  The associated variety of $X$ is the support of $\gr(X)$, which is the variety in $\fg^*$ defined by the ideal annihilating $\gr(X)$.  Since a good filtration of $X$ is $K$-stable, $\fk$ acts trivially on $\gr(X)$, i.e., $\fk\subset\ann(\gr(X))$.  It follows that $AV(X)$ is a $K$-stable subset of $(\fg/\fk)^*$.  In fact, $AV(X)\subset \cN_\gt$, so is a union of (closures of) some $K$-orbits in $\cN_\gt=\Cal N\cap (\fg/\fk)^*$.  We write 
\begin{equation}\label{eqn:av}
AV(X)=\bar{\cO}_1\cup\dots\cup\bar{\cO}_m,
\end{equation}
where the orbits in the expression are maximal in the sense that they are not contained in the closures of others that appear.  If $X$ is irreducible, then each $\cO_k$ appearing in (\ref{eqn:av}) is a `real form' of a single  $G$-orbit in $\Cal N$, that is, $G\cdot\cO_k=\cO^\C$, for all $k=1,\dots,m$, for some $\cO^\C\subset\Cal N$.  It is a fact that in this situation the associated variety of the annihilator of $X$ is $\bar{\cO^\C}$ (\cite[Thm.~8.4]{Vogan91}).  

The associated cycle of a Harish-Chandra module is a formal integer combination
\begin{equation*}
AC(X)=\sum_{k=1}^m m_k\bar{\cO}_k,
\end{equation*}
where $m_k$, the \emph{multiplicity} of $\bar{\cO}_k$ in $AC(X)$, is the rank of $\gr(X)$ at a generic point in $\bar{\cO}_k$.  A thorough treatment of the definition and these facts is contained in \cite{Vogan91}.

Our proof of Theorem A requires that the definition of the associated cycle be extended to the Grothendieck group $\bK(\fg,K)$.  This is not immediate since the associated cycle is not additive on exact sequences as the following example shows.

\begin{Example} Consider $G_\R=SL(2,\R)$.  There is an exact sequence 
$$0\to\C\to \Ind_{P}^{G_\R}(\C)\to X_+\oplus X_-\to 0
$$
where $X_\pm$ are the Harish-Chandra modules of the two discrete series representations of infinitesimal character $\rho$ and $\Ind_{P}^{G_\R}(\C)$ is the (normalized) principal series representation induced from the trivial representation of the minimal parabolic subgroup $P$.   Then 
$$AC(\C)=1\cdot\{0\}\text{ and } AC(\Ind_{P}^{G_\R}(\C))=AC(X_+\oplus X_-)=1\cdot\bar{\cO}_++ 1\cdot\bar{\cO}_-,
$$
where $\cO_\pm$ are the two one-dimensional $K$-orbits in $\cN_\gt$.  Thus, for example, the associated cycle of $\C$ cannot be recovered from the associated cycles of $\Ind_{P}^{G_\R}(\C)$ and $X_+\oplus X_-$.
\end{Example}

The definition of associated variety and associated cycle on $\bK(\fg,K)$ will be in terms of the $K$-equivariant $\textup{K}$-theory of $\cN_\gt$. 

Suppose for a moment that $K$ is any algebraic group acting on a variety $X$.  Denote by $\textup{K}^K(X)$ the Grothendieck group of the (abelian) category of $K$-equivariant coherent sheaves on $X$.  See, for example, \cite{ChrissGinzburg97} and \cite{Thomason87}  for generalities on $K$-equivariant $\textup{K}$-theory.

An important special case occurs when $K$ acts with just one orbit:  $X=K\cdot x=K/K^x$,  $K^x$ the stabilizer of $x$.  Then it is a fact that the $K$-equivariant coherent sheaves are exactly the sheaves of algebraic sections of the homogeneous vector bundles (e.g., \cite[Lemma 2.1.3]{Achar01}).  Recall that the homogeneous vector bundles are constructed from the representations $(\tau,E_\tau)$ of $K^x$ and are of the form $K\underset{K^x}{\times}E_\tau$, which is the quotient of $K\times E_\tau$ by the equivalence relation given by $(kh,v)\sim(k,\tau(h)v)$, for $k\in K,h\in K^x$ and $v\in E_\tau$.  The space of global sections is 
\begin{align*}
\gG(K\underset{K^x}{\times}E_\tau)&=\{\varphi:K\to E_\tau \,|\, \varphi(kh)=\tau(h^{-1})\varphi(k),  k\in K, h\in K^x\}  
=\Ind_{K^x}^K(\tau).
\end{align*}
Since an exact sequence of $K^x$-representations gives an exact sequence of sheaves of local sections of homogeneous vector bundles, we obtain a homomorphism  $\Rep(K^x)\to \textup{K}^K(K/K^x)$.  This map in fact gives an isomorphism:
\begin{equation*}\textup{K}^K(K/K^x)\simeq \Rep(K^x).
\end{equation*}

Returning to an arbitrary action of a group $K$, we conclude that if $\cO=K\cdot x$ is an orbit of $K$ on $X$ and $\cE_\cO(\tau)$ denotes the sheaf of sections of $K\underset{K^x}{\times}\tau\to\cO$, then 
\begin{equation}\label{eqn:vb}
\{[\cE_\cO(\tau)]:\tau\in(K^x)^{\dual}\}
\end{equation}
is a $\Z$-basis of  $\textup{K}^K(\cO).$
One may also see that if $K$ acts on a vector space $V$, then 
\begin{equation*}
\textup{K}^K(V)\simeq \Rep(K).
\end{equation*}
This isomorphism is given by associating to $(\gs,E_\gs)\in K^{\dual}$, the sheaf of sections of the trivial bundle $V\times E_\gs$, that is, $\cO(V)\otimes E_\gs$.

Now suppose that $Z\subset X$ is $K$-stable and $i:Z\hookrightarrow X$ is a closed embedding.  Let $U:=X\smallsetminus Z$ and $j:U\hookrightarrow X$.  Then there is an  important exact sequence
\begin{equation}\label{eqn:seq-2}
\cdots\to \textup{K}_1^K(U)\to \textup{K}^K(Z)\overset{i_*}{\to} \textup{K}^K(X)\overset{j^*}{\to}\textup{K}^K(U)\to 0,
\end{equation}
involving higher $\textup{K}$-groups.
The map $i_*$ is extension by zero and $j^*$ is the restriction from $X$ to $U$.   The surjectivity of $j^*$ is essentially the fact that any equivariant coherent sheaf on an open set extends to an equivariant coherent sheaf on the closure of the set.  The exactness at $\textup{K}^K(X)$ is roughly the statement that two coherent sheaves with the same restrictions to $U$ have difference supported in $Z=X\smallsetminus U$.  See \cite[Ch. II, ex. 6.10]{Hartshorne77} for the nonequivariant case and \cite[Theorem 2.7]{Thomason87} for the equivariant case.  The full exact sequence, along with the definition of the higher $\textup{K}$-groups, is in \cite{Quillen73}. 

Now assume that $K$ has a finite number of orbits on $X$.  List these orbits as $\cO_1,\dots,\cO_m$.   We construct a $\Z$-basis of $\textup{K}^K(X)$ as follows.  For each $\cO_k=K\cdot x_k$ and $\tau_{kl}\in(K^{x_k})\dual$, let $\cE_{\cO_k}(\tau_{kl})$ be the sheaf of local sections of  $K\underset{K^{x_k}}{\times}\tau_{kl}$.  This $K$-equivariant sheaf on $\cO_k$ has an extension to a coherent $K$-equivariant sheaf on $\bar{\cO}_k$.   Extending this sheaf from $\bar{\cO}_k$ to $X$ (by zero), we obtain a  coherent $K$-equivariant sheaf $\tilde{\cE}_{\cO_k}(\tau_{kl})$ on $X$ extending the vector bundle $\cE_{\cO_k}(\tau_{kl})$.  This is the surjectivity of $j^*$ (more precisely,  it is the fact that the sheaf extends to an actual sheaf).  Note that this extension is not necessarily unique.  

One easily shows, by induction on the number of orbits,  that 
\begin{equation}\label{eqn:basis}
\{[\tilde{\cE}_{\cO_k}(\tau_{kl})]\,:\, k=1,\dots,m,\,\tau_{kl}\in(K^{x_k})\dual\}
\end{equation}
spans $\textup{K}^K(X)$ as a $\Z$-module.

In fact we have the the following.
\begin{Thm}\label{thm:basis}
$\{[\tilde{\cE}_{\cO_k}(\tau_{kl})]\,:\, k=1,\dots,m,\tau_{kl}\in(K^{x_k})\dual\}$
 is a $\Z$-basis of $\textup{K}^K(X)$.
\end{Thm}
A version of this is stated in \cite[Lecture 7]{Vogan98} in somewhat different language.  We give the proof in the appendix.

Now return to our pair $(G,K)$ of simple group and fixed points of an involution.  Then the action of $K$ on $\cN_\gt$ has a finite number of orbits $\cO_k=K\cdot x_k,k=1,\dots,m$, and the theorem  gives a $\Z$-basis of $\textup{K}^K(\cN_\gt)$.

We may now give a definition of associated variety and associated cycle of a virtual Harish-Chandra $(\fg,K)$-module as follows.  Consider 
$$
\gr:\cM(\fg,K)\to \cM(S(\fg),K;\cN_\gt),
$$ the target being those finitely generated $(S(\fg),K)$-modules supported in $\cN_\gt$.  This gives a map $\cM(\fg,K)\to \bK(S(\fg),K;\cN_\gt)$ which can be shown to be additive on exact sequences (and independent of good filtration), and thus defines a homomorphism
\begin{equation*}
\bK(\fg,K)\to \bK(S(\fg),K;\cN_\gt).
\end{equation*}
In fact this map is surjective and its kernel is the subgroup of virtual modules having each $K$-type occurring with multiplicity zero  (\cite[Lecture 7]{Vogan98}).  Coherent sheaves on an affine variety are precisely the finitely generated modules over the coordinate ring (\cite[Ch.~II, \S5]{Hartshorne77}).  Taking into account the $K$-actions we get $\bK(S(\fg),K;\cN_\gt)\simeq \textup{K}^K(\cN_\gt)$, giving a surjection 
\begin{equation}\label{eqn:surj}
\GR:\bK(\fg,K)\to \textup{K}^K(\cN_\gt).
\end{equation}

Now fix a basis $\{[\tilde{\cE}_{\cO_k}(\tau_{kl})]: k=1,\dots,m \text{ and }\tau_{kl}\in (K^{x_k})\dual\}$ of $\textup{K}^K(\cN_\gt)$ as in Theorem \ref{thm:basis}.
For $X\in\bK(\fg,K)$, write 
\begin{equation}\label{eqn:expr}
\GR(X)=\sum_{k,l}n_{kl}[\tilde{\cE}_{\cO_k}(\tau_{kl})].
\end{equation}
\begin{Def}  The \emph{associated variety} and \emph{associated cycle} of $X\in\bK(\fg,K)$ are defined by:
\begin{align*}
&AV(X):=\bigcup_{\cO_k\text{ max'l}}\bar{\cO}_k  \\
&AC(X):=\sum_{\cO_k\text{ max'l}}\left(\sum_j n_{kl}\dim(\tau_{kl})\right)\bar{\cO}_k  
\end{align*}
By `$\cO_k$ maximal' we mean that $\cO_k$ occurs in (\ref{eqn:expr}) with coefficient $n_{kl}\neq0$, for some $l$, and $\cO_k$ is not in the closure of another orbit that occurs in (\ref{eqn:expr}) (with a nonzero coefficient).
\end{Def}
The fact that $AV(X)$ and $AC(X)$ are independent of the extensions used to define each $\tilde{\cE}_{\cO_k}(\tau_{kl})$ is contained in Corollary \ref{cor:lt} of the appendix.  When $X$ is a Harish-Chandra module, this definition coincides with the original definition of \cite{Vogan91}.  Note that the multiplicity of $\tilde{\cE}_{\cO_k}(\tau_{kl})$  along $\cO_k$ coincides with that of $\cE_{\cO_k}(\tau_{kl})$, which is the dimension of $\tau_{kl}$.


\section{Dirac Index}
The definition and numerous properties of the Dirac index of a Harish-Chandra module is contained in \cite{MehdiPandzicVogan15}.  The Dirac index is related to the Dirac cohomology, but has some additional nice properties.  For example, it is defined on virtual Harish-Chandra modules (i.e., on the Grothendieck group) and it is well-behaved with respect to tensoring by finite dimensional $\cU(\fg)$-modules.  The Dirac index (as well as the Dirac cohomology) is a virtual $\widetilde{K}$-module\footnote{$\widetilde{K}$ is the pullback of the double cover $\pi:\Spin (\fp)\to SO(\fp)$ by the adjoint representation of $K$, that is $\widetilde{K}=\{(k,s)\in K\times \Spin (\fp) : \Ad(k)=\pi(s)\}$.}. Here we review the definition and the most relevant properties.  

Assume that $\rank(G)=\rank(K)$ and fix a Cartan subalgebra $\fh$ of $\fg$ contained in $\fk$.  In this case $\fp$ is even dimensional, so the spin representation of $\Spin (\fp)$ decomposes into the direct sum of two subrepresentations $S^\pm$.  Recall that $S$ may be constructed by first choosing a maximal isotropic subspace $V\subset \fp$.  Then the Clifford algebra $C(\fp)$ acts on $\wedge(V)$.  Under this action the odd part of $C(\fp)$ sends $\wedge^{\even} (V)$ to $\wedge^{\odd} (V)$, so the even part preserves both $\wedge^{\even} (V)$ and $\wedge^{\odd} (V)$.  Since $\f{so}(\fp)\subset C^{\even} (\fp)$, the  spin representation of $\widetilde{K}$ decomposes into the sum of $S^+:=\wedge^{\even} (V)$ and $S^-:=\wedge^{\odd} (V)$.

Typically one makes a choice of maximal isotropic subspace as follows.  First choose a system of positive roots $\gD^+\subset \gD(\fh,\fg)$ and let $\gD_n^+=\gD^+\cap\gD(\fp)$.  Then 
\begin{equation}\label{eqn:maxiso}
V=\fn_\fp=\sum_{\gb\in\gD_n^+}\fg^{(\gb)},
\end{equation}
the sum of the positive root spaces in $\fp$, is a maximal isotropic subspace of $\fp$.  In this case the $\fh$-weights of $S^\pm$ are
\begin{equation}\label{eqn:wts}
\begin{split}
  &\gD(S^+)=\{\sum_{\gb\in I}\gb -\rho_n : I\subset\gD_n^+, \#I\text{ is even}\} \\
  &\gD(S^-)=\{\sum_{\gb\in I}\gb -\rho_n : I\subset\gD_n^+, \#I\text{ is odd}\}. 
\end{split}
\end{equation}

For a Harish-Chandra module $X$, the Dirac operator 
$$D:X\otimes S\to X\otimes S
$$
is defined by 
$$D=\sum_i \xi_i\otimes c(\xi_i),$$ 
where $\{\xi_i\}$ is any orthonormal basis of $\fp$ and $c(\xi_i)$ is clifford multiplication by $\xi_i$.  The Dirac cohomology is 
$$H_D(X)=\Ker(D)/\Ker(D)\cap\rm{Im}(D),$$
a finite dimensional $\widetilde{K}$-representation. 

For $\xi\in \fp$, $c(\xi)$ sends $S^{\pm}$ to $S^{\mp}$, so (under our equal rank assumption) we also have the $\widetilde{K}$-maps
$$D^{\pm}: X\otimes S^{\pm}\to X\otimes S^{\mp}.$$
One immediately has 
$$H_D(X)=\Ker(D^+)/\Ker(D^+)\cap\rm{Im}(D^-)\,\oplus \,\Ker(D^-)/\Ker(D^-)\cap\rm{Im}(D^+).
$$
\begin{Def}  If $X$ is a Harish-Chandra module, then the \emph{Dirac index} of $X$ is 
$$DI_v(X):=\Ker(D^+)/\Ker(D^+)\cap\rm{Im}(D^-) - \Ker(D^-)/\Ker(D^-)\cap\rm{Im}(D^+),
$$
a virtual $\widetilde{K}$-representation.
\end{Def}

We now list some key properties of the Dirac index.  Let $X$ be a Harish-Chandra $(\fg,K)$-module.

\begin{Prop}\upshape{(\cite{HuangPandzic02})} \label{prop:11}
If $X$ has infinitesimal character of Harish-Chandra parameter $\gl$, then each irreducible constituent of $DI_v(X)$ has infinitesimal character of the form $w(\gl)$, for some $w\in W$.
\end{Prop}

\begin{Prop}\upshape{(\cite[Prop.~3.2]{MehdiPandzicVogan15})} \label{prop:22}
$DI_v(X)=X\otimes S^+-X\otimes S^-$.  Therefore, the Dirac index depends only on the $K$-types of $X$.
\end{Prop}

\begin{Prop}\upshape{(\cite[Cor.~3.3]{MehdiPandzicVogan15})} \label{prop:33}
The Dirac index is additive on exact sequences, so is well-defined on virtual modules, giving a homomorphism
$$DI_v:\bK(\fg,K)\to \Rep(\widetilde{K}).
$$
\end{Prop}

\begin{Prop}\upshape{(\cite[Prop.~5.2]{MehdiPandzicVogan15})} \label{prop:44}
 Let $\gL$ be the weight lattice in $\fh^*$ and suppose that $\{X(\gl)\}_{\gl\in\gl_0+\gL}$ is a coherent family of virtual Harish-Chandra modules, then $\dim(DI_v(X(\gl)))$ is a $W$-harmonic polynomial in $\gl$.
\end{Prop}

Suppose $X$ has regular infinitesimal character of Harish-Chandra parameter $W\cdot\gl_0$ with $\gl_0$ dominant.  Then there is a unique coherent family $\{X(\gl)\}$ with $X=X(\gl_0)$ (see, for example, \cite[Ch. 7, \S2]{Vogan81}).

\begin{Def}  The \emph{Dirac index polynomial} is the $W$-harmonic polynomial defined by $DI_p(X)=\dim(DI_v(X(\gl)))$.
\end{Def}
\noindent By the Weyl dimension formula, the Dirac index polynomial is homogeneous of degree $\#\gD_c^+$.

\begin{Prop}\upshape{(\cite[Prop.~5.6]{MehdiPandzicVogan15})} \label{prop:55}
If the Gelfand-Kirillov dimension of $X$ is less than $\#\gD^+-\#\gD_c^+$, then $DI_p(X)=0$.
\end{Prop}


\section{Cohomological Induction}
Cohomologically induced representations play an important role in the proof of Theorem A.   In \S\ref{sec:ext} extensions of the sheaves $\cE_{\cO_k}(\tau_{kl})$ will be given a somewhat explicit form using  cohomological induction.  
We will also use a formula for the Dirac index of discrete series representations.

Let $\fq=\fl+\fu$ be a $\gt$-stable parabolic subalgebra of $\fg$ and $Q=LU$ the corresponding parabolic subgroup of $G$.  The cohomological induction functors (as defined in \cite{Vogan81} and \cite{KnappVogan95}) will be denoted by 
$$
\cR_\fq^i:\cM(\fl,K\cap L)\to\cM(\fg,K).
$$
The following are a few of the most basic properties.  Let $Z\in\cM(\fl,K\cap L)$.

\begin{enumerate}[(a)] 
\item  $\cR_\fq^i(Z)=0$, unless $0\leq i\leq s$, where $s:=\dim_\C(K/K\cap Q)$.  
\item If $Z$ has infinitesimal character $\gl_Z$, then $\cR_\fq^i(Z)$ has infinitesimal character $\gl_Z+\rho(\fu)$. 
\item If $Z$ is irreducible of infinitesimal character $\gl_Z$ and $\IP{\gl_Z+\rho(\fu)}{\gb}>0,$ for all $\gb\in\gD(\fu)$, then $\cR_\fq^s(Z)$ is irreducible and $\cR_\fq^i(Z)=0$, for $i\neq s$.  
\item In the special case when $\fq=\fb=\fh+\fn$, a $\gt$-stable Borel subalgebra, and $\rank(G)=\rank(K)$, $\cR_\fb^s(\C_\gl)$ is the Harish-Chandra module of a discrete series representation when $\gl$ is dominant (for $\gD^+=\gD(\fn)$) and analytically integral.  In this case the irreducible representation $E_\mu$ of $K$ having highest weight $\mu=\gl+2\rho_n$ (with respect to $\gD_c^+=\gD_c\cap \gD(\fn)$) occurs  with multiplicity one as a $K$-type in $\cR_\fb^s(\C_\gl)$.  All other $K$-types have highest weights of the form $\displaystyle \gl+2\rho_n+\sum_{\gb\in\gD_n^+}n_\gb\gb, \,n_\gb\geq 0$.  All discrete series representations occur in this way.  
\item Since a short exact sequence in $\cM(\fl,K\cap L)$ gives a long exact sequence of the cohomologically induced representations, the alternating sum $\sum_i(-1)^i\cR_\fq^{s-i}(Z)$ 
depends only on 
the class of $Z$ in $\bK(\fl,K\cap L)$.  We may therefore make the following definition.
\end{enumerate}
\begin{Def} Let $\cR_\fq:\bK(\fl,K\cap L)\to \bK(\fg,K)$ be defined by
$$
\cR_\fq([Z])=\sum_i(-1)^{i}[\cR^{s-i}_\fq(Z)].
$$
\end{Def}  
\begin{enumerate}[(f)]
\item  Let $\gL$ be the integral lattice in $\fh^*$.  If $\{Z(\gl)\}_{\gl\in\gl_0+\gL}$ is a coherent family of virtual $(\fl,K\cap L)$-modules, then 
$$\left\{X_\fq(\gl)=\cR_\fq\left([Z(\gl-\rho(\fu))] \right) \right\}_{\gl\in\gl_0+\rho(\fu)+\gL},
$$
is a coherent family of virtual $(\fg,K)$-modules.  See \cite[Cor 7.2.10]{Vogan81}.
\end{enumerate}

Of particular importance to us is the Blattner formula for the $K$-decomposition of cohomologically induced representations.  (See, for example, \cite[Thm. 6.3.12]{Vogan81}.)   The Blattner formula may be expressed as follows.  Suppose  $Z$ is a  Harish-Chandra $(\fl,K\cap L)$ module.  Consider irreducible $K\cap L$-representations $(\pi,F_\pi)$.  Then $\dim ( 
\Hom_K(\pi,Z))$ is the multiplicity of $\pi$ in $Z$ and
\begin{equation}\label{eqn:blattner}
\Cal R_\fq(Z)|_K = \sum_\pi \sum_{d=0}^\infty\dim (
\Hom_K(\pi,Z))\sum_q(-1)^q H^{s-q}\left(K/K\cap Q, F_\pi\otimes\C_{2\rho(\fu)}\otimes S^d(\fp\cap\fu)\right).
\end{equation}
In this formula $H^{s-q}\left(K/K\cap Q, F\right)$ denotes the sheaf cohomology of the sheaf of sections of the homogeneous vector bundle on $K/K\cap Q$ defined by the $K\cap Q$-representation $F_\pi$ ($K\cap U$ acting trivially).  This is computed by the Borel-Weil Theorem.  For example, if $\mu$ is dominant for a positive system $\gD_c^+$ containing $\gD_c(\fu)$ and $F_\mu$ is the irreducible representation or $K\cap L$ of highest weight $\mu$, then $H^{s-q}\left(K/K\cap Q, F_\mu\right)\simeq H^0(K/K\cap \bar{Q},F_\mu\otimes \C_{-2\rho_c(\fu)})$ is the irreducible representation of $K$ having highest weight $\mu-2\rho_c(\fu)$.  We also observe that the alternating sum of cohomology is well-defined on the Grothendieck group $\Rep(K\cap Q)=\Rep(K\cap L)$.  

Consider a Borel subalgebra $\fb=\fh+\fn$ (containing our compact Cartan subalgebra $\fh$) and set $\gD^+=\gD(\fn)$.   Define $\X(\gl)=\cR_\fb(\C_{\gl-\rho})$, giving the coherent family containing discrete series representations associated to $\fb$.  By (\ref{eqn:blattner}) the  $K$-decomposition of discrete series representations is therefore given by  
\begin{subequations}\label{eqn:bf}
\begin{align}\label{eqn:bf-1}
 \X(\gl)|_K&=\sum_{q=0}^s(-1)^q\sum_{d=0}^\infty H^{s-q}\left(K/K\cap B,\C_{\gl+\rho}\otimes S^d(\fn\cap\fp)\right)  \\  \label{eqn:bf-2}
 &=\sum_{q=0}^s(-1)^q\;\sum_{d=0}^\infty \;\;\sum_{\gd\in\gD(S^d(\fn\cap\fp))} H^{s-q}\left(K/K\cap B,\C_{\gl+\rho+\gd}\right).
\end{align}
\end{subequations}

The Dirac index of discrete series representations may now be computed.  Recall that we have defined the Weyl dimension polynomial by 
$$P_K(\gl)=\prod_{\ga\in\gD_c^+}\frac{\IP{\gl}{\ga}}{\IP{\rho_c}{\ga}}.
$$
Since $\{X_\fb(\gl)\}$ is a coherent family, $DI_p(X_\fb(\gl))$ is a harmonic polynomial (Prop.~\ref{prop:44}).  The following well-known result is implicit in \cite{Parthasarathy72} and follows
from the main result of \cite{HuangPandzic02}.

\begin{Prop}\label{prop:di-ds}Assume that $S^\pm$ are defined (as in (\ref{eqn:maxiso})) by the maximal isotropic subspace determined by $\gD^+=\gD(\fn)$.  Then $\displaystyle{ DI_p(\X(\gl))=P_K(\gl)}$, where the positive system of compact roots appearing in the definition of $P_K(\gl)$ is $\gD_c^+=\gD_c\cap\gD^+$.
\end{Prop}
\begin{proof} Writing $B=HN$, we have
$$
S^+-S^-= \sum_j (-1)^j \wedge^j(\fn\cap\fp)\otimes \C_{-\rho_n}
$$
as $H$-representations.  Tensoring by $S(\fn\cap\fp)$, gives
\begin{equation}\label{eqn:fk}
S(\fn\cap\fp)\otimes(S^+-S^-)=   \sum_j (-1)^j S(\fn\cap\fp) \otimes \wedge^j(\fn\cap\fp)\otimes \C_{-\rho_n}.
\end{equation}
Now the Koszul resolution
\begin{equation}\label{eqn:koszul}\cdots\to\wedge^2(\fn\cap\fp)\otimes S(\fn\cap \fp)   \to  \wedge^1(\fn\cap\fp)\otimes S(\fn\cap \fp)\to S(\fn\cap \fp) \to \C\to 0,  
\end{equation}
an exact sequence of $B\cap K$-representations, says the right-hand side of (\ref{eqn:fk}) is $\C_{-\rho_n}$ and we get
\begin{equation}\label{eqn:kz}
S(\fn\cap\fp)\otimes(S^+-S^-)=  \C_{-\rho_n}.
\end{equation}
  
The Dirac index is 
\begin{align*}
X_\fb(\gl)\otimes S^+&-X_\fb(\gl)\otimes S^-  =X_\fb(\gl)\otimes (S^+- S^-)  \\
  & = \sum_q(-1)^q H^{s-q}\left(K/K\cap B, \C_{\gl+\rho}\otimes S(\fn\cap\fp)\otimes (S^+-S^-)\right)  \\
  & = \sum_q(-1)^q H^{s-q}\left(K/K\cap B, \C_{\gl+\rho_c}\right),\text{ by (\ref{eqn:kz})} .
\end{align*}
When $\gl$ is $\gD_c^+$-dominant, this is the finite dimensional representation of $K$ with highest weight $\gl-\rho_c$, which has dimension $P_K(\gl)$.  Therefore, the two polynomials $P_K(\gl)$ and $DI_p(X_\fb(\gl))$ coincide.
\end{proof}
\begin{Rem}\label{rem:sign} If a different positive system $\widetilde{\gD^+}$ is used to define the maximal isotropic subspace $\widetilde{\fn}_\fp$ determining $\widetilde{S^\pm}$ as in the discussion surrounding (\ref{eqn:maxiso}), then the Dirac index polynomial of $X_\fb(\gl)$ will be $\pm P_K(\gl)$.  The sign is determined as follows.  If $S^\pm$ denotes the choice determined by $\gD^+$, then 
\begin{equation*}
 \widetilde{S^+}=\begin{cases}
    S^+,&\text{if }\#(\widetilde{\gD_n^+}\smallsetminus \gD_n^+)\text{ is even}  \\
    S^-,&\text{if }\#(\widetilde{\gD_n^+}\smallsetminus \gD_n^+)\text{ is odd}, 
    \end{cases}
\end{equation*}
since the weight $-\rho_n$ occurs in $\widetilde{S^+}$ if and only if $\#(\widetilde{\gD_n^+}\smallsetminus \gD_n^+)$ is even.  Thus the corresponding Dirac index is $\widetilde{DI}=(-1)^{\#(\widetilde{\gD_n^+}\smallsetminus \gD_n^+)}DI$.  Therefore,
\begin{align*}
  DI_p(X_{\tilde{\fb}}(\gl))&=(-1)^{\#(\widetilde{\gD_n^+}\smallsetminus \gD_n^+)}\widetilde{DI}_p(X_{\tilde{\fb}}(\gl))  \\
  &=(-1)^{\#(\widetilde{\gD_n^+}\smallsetminus \gD_n^+)}\widetilde{P}_K(\gl), \text{ $\widetilde{P}_K(\gl)$ defined by }\widetilde{\gD_c^+},  \\
  &=(-1)^{\#(\widetilde{\gD^+}\smallsetminus \gD^+)}P_K(\gl).
\end{align*}  
\end{Rem}

The Blattner  formula may be `inverted' to give an expression for an irreducible $K$-representation as a virtual $(\fg,K)$-module.  This will  allow us (in \S\ref{sec:ext}) to give expressions for extensions of the sheaves $\cE_\cO(\tau)$.

\begin{Prop}\label{lem:bi}  Let $\fb$ be a Borel subalgebra as above and let $\gg\in\fh^*$ be analytically integral.  Then 
\begin{equation}\label{eqn:bi-1}
\sum_q(-1)^qH^{s-q}(K/K\cap Q,  \C_{\gg+2\rho_c})=\sum_{A\subset\gD_n^+}(-1)^{\#A} \X(\gg+\rho_c-\rho_n+2\rho(A))|_K.
\end{equation}
In particular, if $\gg$ is $\gD_c^+$-dominant then the irreducible representation of $K$ having highest weight $\gg$ may be written as
\begin{equation}\label{eqn:bi}
E_\gg=\sum_{A\subset\gD_n^+}(-1)^{\#A} \X(\gg+\rho_c-\rho_n+2\rho(A))|_K.
\end{equation}
\end{Prop}
\begin{proof}
Using (\ref{eqn:bf}), the right-hand side of (\ref{eqn:bi-1}) is 
\begin{equation*}
\sum_{A\subset\gD_n^+}(-1)^{\#A}\sum_{q=0}^s(-1)^q\sum_{d=0}^\infty H^{s-q}(K/K\cap B, \C_{\gg+2\rho_c+2\rho(A)}\otimes S^d(\fn\cap\fp)). 
\end{equation*}
Using the Koszul resolution (\ref{eqn:koszul}), we have (as virtual $H$-representations)
\begin{equation*}
  \sum_{A\subset\gD_n^+}(-1)^{\#A}\sum_{d=0}^\infty\C_{\gg+2\rho_c+2\rho(A)}\otimes S^d(\fn\cap\fp)=\C_{\gg+2\rho_c},
\end{equation*}
and the statement follows.  The formula for $E_\gg$ follows from the Borel-Weil Theorem,
when $\gg $ is $\gD_c^+$-dominant.
\end{proof}


\section{An extension}\label{sec:ext}

In this section we give extensions,  from $\cO$ to $\cN_\gt$, of the coherent sheaves $\cE_\cO(\tau)$ for all $K$-orbits in $\cN_\gt$.  These extensions will be constructed as pushforwards of vector bundles under a resolution of singularities of $\bar{\cO}$.

Let $\cO=K\cdot x$ be our arbitrary $K$-orbit in $\cN_\gt$.  Identifying $\cN_\gt$ with $\cN\cap \fp\subset\fg$, $x$ fits into a standard triple $\{x,h,y\}$ with semisimple element $h$ lying in $\fh$.  The eigenvalues of $\ad(h)$ are integers, therefore give a grading of $\fg$:
$$\fg_m=\{X\in\fg: [h,X]=mX\},\, m\in\Z.
$$
We set $\fp_m=\fg_m\cap \fp$, the $m$-eigenspace in $\fp$.  The grading gives a $\gt$-stable parabolic subalgebra 
$$\fq=\fl+\fu, \text{ with } \fl=\fg_0 \text{ and } \fu=\sum_{m\leq -1}\fg_m.
$$  
Note that $\gD(\fu)=\{\ga\in\gD : \ga(h)<0\}$.  The corresponding parabolic subgroup in $G$ and Levi decomposition are written as $Q=LU$.

There are two well-known facts that play an important role for us (\cite{KostantRallis71}):
\begin{enumerate}[(i)]  
\item The stabilizer $K^x$ has a Levi decomposition with reductive part $(K^x)_{\rm red}=K\cap L^x=(K\cap L)^x$ and $K^x\subset K\cap\bar{Q}$.  
\item If we set $\fp[2]:=\sum_{m\geq 2}\fp_m$, then $\fp[2]$ is stable under $\Ad(K\cap\bar{Q})$ and the morphism 
\begin{equation}\label{eqn:resolution}
 \mu:K\underset{K\cap \bar{Q}}{\times}\fp[2]\to\bar{\cO},\, \, \mu(k,\xi):=k\cdot\xi,
\end{equation}
 is a resolution of singularities.
\end{enumerate}
The higher direct images $R^i\mu_*$ of the morphism $\mu$ give a homomorphism of $\textup{K}^K$ groups: 
\begin{equation}\label{eqn:push-fwd}
\begin{split}
&\mu_*:\textup{K}^K(K\underset{K\cap\bar{Q}}{\times}\fp[2])\to \textup{K}^K(\cN_\gt) \\
&\mu_*([\cS]):=\sum_i(-1)^i[R^i\mu_*(\cS)].
\end{split}
\end{equation}
If $\cS$ is a coherent sheaf on $K\underset{K\cap\bar{Q}}{\times}\fp[2]$, then each higher direct image $R^i\mu_*(\cS)$ is a coherent sheaf on $\bar{\cO}$, as $\mu$ is proper.  We view the right-hand side of (\ref{eqn:push-fwd}) as in $\textup{K}^K(\cN_\gt)$ by extension by zero from $\bar{\cO}$ to $\cN_\gt$.

Let $\gs$ be a representation of $K\cap\bar{Q}$.  Then there is a $K$-equivariant vector bundle $K\underset{K\cap\bar{Q}}{\times}(\fp[2]\times\gs)\to K\underset{K\cap\bar{Q}}{\times}\fp[2]$; write $\cS_\gs$ for its sheaf of algebraic sections.

\begin{Prop}\label{prop:ext} \upshape{(1)} $[\mu_*(\cS_\gs)]$ is an extension of $\cE_\cO(\gs|_{K^x})$. \\
\upshape{(2)} $\gG(\mu_*(\cS_\gs))|_K=\sum_j(-1)^j\cR_\fq(\gs\otimes\wedge^j\fp_1^* \otimes\C_{-2\rho(\fp\cap\fu)})|_K$
\end{Prop}
\begin{Rem} The terms $\cR_\fq(\gs\otimes \wedge^j\fp_1^*\otimes\C_{-2\rho(\fp\cap\fu)})$ of the sum in (2) requires some explanation.  We view the (finite dimensional representations) $\gs \otimes \wedge^j\fp_1^*\otimes\C_{-2\rho(\fp\cap\fu)}$ as virtual $(\fl,K\cap L)$-modules.  This may be accomplished by using equation (\ref{eqn:bi}) of Proposition \ref{lem:bi} applied to $(\fl,K\cap L)$-modules.
\end{Rem}

\begin{proof} In \cite[\S2]{Achar01} it is shown that $\mu_*(\cS_\gs)$ is an extension of $[\cE_\cO(\gs|_{K^x})]$.  The space of sections (as $K$-module) is computed  in \cite[\S2]{Achar01}:
\begin{equation*}
\gG(\mu_*(\cS_\gs))=\sum_q(-1)^q\sum_{d=0}^\infty H^q\left(K/K\cap\bar{Q},\gs\otimes S^d(\fp[2]^*)\right).
\end{equation*}
By the Koszul resolution, $\C=\sum_j(-1)^j\wedge^j\fp_1^*\otimes S(\fp_1^*)$ as virtual $K\cap L$-representations.  Therefore,
\begin{align*}
\gG(\mu_*(\cS_\gs))&=\sum_q(-1)^q\sum_{d=0}^\infty \sum_j(-1)^j H^q\left(K/K\cap\bar{Q},\gs\otimes \wedge^j\fp_1^*\otimes  S^d((\fp\cap\bar{\fu})^*)\right)   \\
&=\sum_q(-1)^q\sum_{d=0}^\infty \sum_j(-1)^j H^q\left(K/K\cap\bar{Q},\gs\otimes \wedge^j\fp_1^*\otimes  S^d(\fp\cap \fu ) \right)   \\
&=\sum_{d=0}^\infty \sum_j(-1)^j \sum_q(-1)^q H^{s-q}\left(K/K\cap Q,\gs\otimes \wedge^j\fp_1^*\otimes  S^d(\fp\cap \fu )\otimes \C_{2\rho_c(\fu)} \right)   
\end{align*}
By equation (\ref{eqn:blattner}), this is precisely the $K$-decomposition of 
\begin{equation*}
\sum_j(-1)^j\cR_\fq(\gs\otimes \wedge^j\fp_1^*\otimes \C_{-2\rho(\fp\cap\fu)}),
\end{equation*}
viewing $\gs\otimes \wedge^j\fp_1^*\otimes \C_{-2\rho(\fp\cap\fu)}$ as the restriction of a (virtual) $(\fl,K\cap L)$-module.
\end{proof}

Part (2) of the proposition may be put into a more useful form for us.   First, it suffices to assume $\gs$ is irreducible with highest weight $\gl_\gs$.  Then
\begin{align*}
  \gs\otimes\wedge^j\fp_1^* & \otimes\C_{-2\rho(\fp\cap\fu)} \\
  & \simeq \sum_{\gl\in\gD(\wedge^j\fp_1^*)}\sum_i(-1)^iH^i(K\cap L/K\cap L\cap \bar{B},\C_{\gl_\gs+\gl-2\rho(\fp\cap\fu)}), \text{ by Borel-Weil},  \\
  &\simeq \;\;\sum_i\sum_{C\subset\gD(\fp_1),\#C=j}(-1)^iH^i(K\cap L/K\cap L\cap \bar{B},\C_{\gl_\gs-2\rho(C)-2\rho(\fp\cap\fu)})  \\  
    &\simeq \;\;\sum_i\sum_{C\subset\gD(\fp_1),\#C=j}(-1)^iH^{s_1-i}(K\cap L/K\cap L\cap B,\C_{\gl_\gs-2\rho(C)-2\rho(\fp\cap\fu)+2\rho_c(\fl)}),  \\
   &\qquad\qquad \text{ where }s_1=\dim(K\cap L/K\cap L\cap B),  \\
  &\simeq\sum_{A\subset\gD_n^+(\fl)}\,\,\sum_{\substack{C\subset\gD(\fp_1) \\ \#C=j\,}}(-1)^{\#A}\XL(\gl_\gs+\rho_c(\fl)-\rho_n(\fl)+2\rho(A)-2\rho(C)-2\rho(\fp\cap\fu)), \\
    &\qquad\qquad\text{ by Lemma \ref{lem:bi} applied to $(\fl,K\cap L)$}.
\end{align*}
(Note: any positive system for $\gD(\fl)$ that contains $\gD_c(\fl)^+=\gD_c^+\cap \gD(\fl)$ may be used here.)  Now substituting into the formula of Part (2) of Proposition \ref{prop:ext} and using induction in stages for cohomological induction ($\cR_\fb=\cR_\fq\circ\cR_{\fb\cap\fl}$) gives the following corollary.
\begin{Cor}\label{cor:ext} For $\cS_\gs$ as above 
$$\displaystyle \gG(\mu_*(\cS_\gs))|_K\simeq \sum_{\substack{A\subset\gD_n^+(\fl) \\ C\subset\gD(\fp_1)}} (-1)^{\#A+\#C}\X(\gl_\gs+\rho_c(\fl)-\rho_n(\fl)+2\rho(A)-2\rho(C)+\rho(\fu)-2\rho(\fp\cap\fu))|_K.$$
\end{Cor}

We have not yet given an extension of $\cE_{\cO}(\tau)$ for an arbitrary $\tau\in(K^x)^{\dual}$.  We prove below that restriction 
$$\Rep(K\cap\bar{Q})\to \Rep(K^x),$$
is a surjection.  This is equivalent to saying the each irreducible representation $\tau$ of $K^x$ extends to a \emph{virtual } representation of $K\cap\bar{Q}$.  (It is typically \emph{not} the case that $\tau$ extends to an honest representation of $K\cap\bar{Q}$.)
Let $\tau\in(K^x)^{\dual}$ and write $\tau=\sum_j n_j\gs_j|_{K^x}$, with $n_j\in\Z$ and $\gs_j\in(K\cap\bar{Q})^{\dual}$.  Now Corollary \ref{cor:ext} tells us how to express an extension $[\tilde{\cE}_{\cO}(\tau)]=\sum_j n_j [\tilde{\cE}_{\cO}(\gs|_{K^x})]$ in terms of discrete series representations.
\begin{Rem} This extension is not canonical since there seems to be no canonical way to extend $\tau$ to a virtual representation of $K\cap\bar{Q}$.
\end{Rem}

The following proposition gives the fact used above.
\begin{Prop}\label{prop:surj}   The restriction map $\Rep(K\cap\bar{Q})\to \Rep(K^x)$ is surjective.
\end{Prop}
This follows from a more general fact.
\begin{Lem} If $(H,V)$ is a prehomogeneous vector space and $H\cdot v$ is the open orbit, then restriction
$$\Rep(H)\to \Rep(H^v)$$
is surjective.
\end{Lem}
\begin{proof} Let $j:H\cdot v\hookrightarrow V$ be the open embedding.  Then 
$$j^*:\textup{K}^H(V)\to \textup{K}^H(H\cdot v)$$
is surjective by (\ref{eqn:seq-2}).    Making the identifications $\textup{K}^H(V)\simeq \Rep(H)$ and $\textup{K}^H(H\cdot v)\simeq \Rep(H^v)$, the map $j^*$ (= restriction of sheaves) corresponds to restriction of representations.  This is because $(\gs,E_\gs)\in H^{\dual}$ corresponds to $\cO(V)\otimes E_\gs\in \textup{K}^H(V)$, so restricts to $\cO(H/H^v)\otimes E_\gs\simeq \Ind_{H^v}^H(E_\gs|_{H^v})$.
\end{proof}

Now Proposition \ref{prop:surj} follows since $(K\cap\bar{Q},\fp[2])$ is a prehomogeneous vector space with open orbit $K\cap\bar{Q}\cdot x$ having stabilizer $ K^x\cap \bar{Q}=K^x$.


\section{Proof of the main theorem}\label{sec:proof}
Continue with the hypothesis that the complex groups $G$ and $K$ have equal rank. Recall that we have fixed (i) a Cartan subalgebra $\fh$ of $\fg$ that is contained in $\fk$ and (ii) a positive system of roots $\Delta_c^+$ in $\Delta(\fh,\fk)$.   The Weyl dimension polynomial for $K$ is 
\begin{equation}\label{eqn:wdp}
P_K(\gl):=\prod_{\ga\in\Delta_c^+}\frac{\IP{\gl}{\ga}}{\IP{\rho_c}{\ga}}.
\end{equation}
   
Consider the $W$-representation $\C[W]\cdot P_K(\gl)$ generated by $P_K(\gl)$.  The Weyl group $W_K$ acts on $P_K(\gl)$ by the sign representation.  No copy of the sign representation occurs in $S(\fh)$ in degree less that $\#\gD_c^+=\deg(P_K(\gl))$.  The sign representation of $W_K$ occurs with multiplicity exactly one in this degree, therefore $\C[W]\cdot P_K(\gl)$ is irreducible (e.g., \cite[Prop. 11.2.3]{Carter85}).
  
 The Springer correspondence associates an irreducible representation of $W$ to each nilpotent $G$-orbit $\cO^\C=G\cdot x$ in $\cN$.  To describe the Springer correspondence, we let  $\fB$ be the flag variety for $G$, $\fB^x=\{\fb\in \fB : x\in \fb\}$ (a `Springer fiber') and  $A(x)=G^x/(G^x)_{\rm e}$ (the `component group' of the stabilizer of $x$).  Then $A(x)$ acts on $\fB^x$, so also acts on the singular cohomology $H^{\bullet}(\fB^x)$.  A linear action of $W$ on $H^{d_x}(\fB^x), d_x:=\dim(\fB^x),$ is constructed in \cite{Springer78}.  This action commutes with the action of $A(x)$, giving a representation of $W$ on the subspace of $A(x)$ invariants.  The $W$-representation on $\left(H^{d_x}(\fB^x)\right)^{A(x)}$ is irreducible.  This gives an assignment, which is in fact one-to-one,  of an irreducible $W$-representation to each complex nilpotent orbit:
$$
\cO^\C=G\cdot x\mapsto \pi(\cO^\C,1):=\left(H^{d_x}(\fB^x)\right)^{A(x)}.
$$
An irreducible representation of $W$ is said to be \emph{Springer} if it is equivalent to some $\pi(\cO^\C,1)$. 
It is an important fact that the symmetric algebra $S(\fh)$ contains a $W$-subrepresentation equivalent to $\pi(\cO^\C,1)$ in degree $d_x$ and contains no such subrepresentation in lower degrees (see \cite{Lusztig79}). 
The setting for our main theorem, Theorem \ref{thm:main} below, is a pair $(G,K)$ for which $\C[W]\cdot P_K(\gl)$ is Springer.  Not all pairs satisfy this requirement.  Tables of \S\ref{sec:const-classical} and \S\ref{sec:const-exceptional} together give a complete list of pairs $(G,K)$ for which the requirement holds.  Also listed on the tables are the (unique) orbits  $\cO^\C$ with $\C[W]\cdot P_K(\gl)=\pi(\cO^\C,1)$.   Since a $W$-representation equivalent to $\C[W]\cdot P_K(\gl)$ occurs in $S^{d_x}(\fh)$, $d_x$ is at most $\#\gD_c^+=\deg(P_K(\gl))$, and  since the sign representation of $W_K$ (a $W_K$-subrepresentation of $\C[W]\cdot P_K(\gl)$) occurs in no degree less that $\#\gD_c^+$, we conclude that $d_x=\#\gD_c^+$.   On the other hand  the general dimension formula for a Springer fiber \cite[Theorem 4.6]{Steinberg76} tells us that 
$$\#\gD_c^+ =d_x=\#\gD^+-\frac12\dim(\cO^\C).$$
Therefore,
\begin{equation}\label{eqn:dimO}
\frac12\dim(\cO^\C)=\#\gD^+-\#\gD_c^+=\#\gD_n^+,
\end{equation}
the number of positive noncompact roots.
 For the remainder of the article we assume that $(G,K)$ is one of the pairs on the tables and $\cO^\C$ is the orbit satisfying 
 \begin{equation}\label{eqn:spr}
 \pi(\cO^\C,1)=\C[W]\cdot P_K(\gl).
\end{equation}

List the real forms of $\cO^\C$ as $\cO_1,\dots,\cO_r$.  These are the $K$-orbits in $\cN_\gt$ for which $G\cdot\cO_k=\cO^\C$.  Now suppose $X$ is an irreducible Harish-Chandra module with $AV(\ann(X))=\bar{\cO^\C}$.  Such Harish-Chandra modules are those for which 
\begin{equation*}
AC(X)=\sum_{k=1}^r m_k(X)\bar{\cO}_k,
\end{equation*}
by \cite[Thm.~8.4]{Vogan91}.
If we interpret $m_k(X)$ to be $0$ when $AV(X)\subset\bar{\cO^\C}\smallsetminus\cO^\C$ (i.e., when $AV(\ann(X))\subset \bar{\cO^\C}\smallsetminus\cO^\C$), then our main theorem may be stated as the following.
\begin{Thm}\label{thm:main}  Let $\cO^\C$ be as in (\ref{eqn:spr}), then there exist integers $c_1,\dots,c_r$ so that 
\begin{equation}\label{eqn:main}
DI_p(X)=\sum_{k=1}^r c_km_k(X),
\end{equation} 
for all $X\in\bK(\fg,K)$ with $AV(\ann(X))\subset \bar{\cO^\C}$.
\end{Thm}
\begin{Rem} For the proof of the theorem we do not need to make a choice for $S^\pm$.  However, we will need to make some (arbitrary) choice when we compute the numerical values of the constants in \S\ref{sec:const-classical} and \S\ref{sec:const-exceptional}.
\end{Rem}

To prove the theorem we may assume that $X$ is irreducible.  This is because the Dirac index is additive and, although the  associated cycle is not additive, we have the following property.  If $AV(\ann(X_j))\subset \bar{\cO^\C}, j=1,2,3$ and $0\to X_1\to X_2\to X_3\to 0$ is exact, then $m_k(X_1)+m_k(X_3)=m_k(X_2)$.  We may also assume that $AV(\ann(X))= \bar{\cO^\C},$ since $AV(\ann(X))\subset\bar{\cO^\C}\smallsetminus \cO^\C$ implies both the Dirac index polynomial is zero (by Prop.~\ref{prop:55}) and $m_k(X)=0$.

We therefore assume that $X$ is irreducible and $AV(\ann(X))=\bar{\cO^\C}$. Write $\GR(X)$ in terms of the basis (\ref{eqn:basis}):
\begin{equation*}
\GR(X) =\sum_{k=1}^r\sum_l n_{kl} [\tilde{\cE}_{\cO_k}(\tau_{kl})]+\sum_{\cO\subset\bar{\cO^\C}\smallsetminus \cO^\C}n_{\cO,\tau}[\tilde{\cE}_\cO(\tau)].
\end{equation*}
Then since the Dirac index is defined on $\textup{K}^K(\cN_\gt)$ and vanishes on $\tilde{\cE}_\cO(\tau)$ when $\cO\subset\bar{\cO^\C}\smallsetminus \cO^\C$, we have
\begin{equation*}
DI_p(X)=\sum_{k=1}^r\sum_l n_{kl} DI_p\left(\tilde{\cE}_{\cO_k}(\tau_{kl})\right).
\end{equation*}
On the other hand 
\begin{equation*}
AC(X)=\sum_{k=1}^r \left( \sum_l n_{kl} \dim(\tau_{kl})\right)\,\bar{\cO}_k.
\end{equation*}
Therefore, Theorem \ref{thm:main} is a consequence of the following.
\begin{Thm}\label{thm:main-2} With $\cO_1,\dots,\cO_r$ as above, there are integers $c_k, k=1,\dots,r$  so that 
\begin{equation}\label{eqn:2}
DI_p\left(\tilde{\cE}_{\cO_k}(\tau)\right)=c_k\dim(\tau)
\end{equation}
for all $\tau\in(K^{x_k})^{\dual}$.
\end{Thm}

To prove this theorem we take an arbitrary $\cO_k=\cO=K\cdot x$ and $\tau\in {K^x}^{\dual}$ and extend $\tau$ to a virtual representation $\gs$ of $K\cap L$.  We may assume that $\gs$ is irreducible, then use the explicit extension given in Proposition \ref{prop:ext}.  The formula for the Dirac index of the discrete series will then relate the two sides of (\ref{eqn:2}).

Therefore it suffices to prove that for each $K$-orbit in $\cN_\gt$ there is a constant $c$ so that 
\begin{equation*}
DI_p\left(\mu_*(\cS_\gs)\right)=c\hspace{.5pt}\dim(\gs), \text{ for all }\gs\in (K\cap L)\dual.
\end{equation*}

Letting $\gl_\gs$ be the highest weight of $\gs$, by Corollary \ref{cor:ext} we need a constant $c$ so that $c\hspace{.5pt}\dim(\gs)$ is 
\begin{equation*}
  \sum_{\substack{A\subset\gD_n^+(\fl)  \\  C\subset\gD(\fp_1)}} (-1)^{\#A+\#C}DI_p\left(\X(\gl_\gs+\rho_c(\fl)-\rho_n(\fl)+2\rho(A)-2\rho(C)+\rho(\fu)-2\rho(\fp\cap\fu))\right).  
\end{equation*}

Since the formula for the Dirac index polynomial of a discrete series representation is given by $DI_p(\X(\gl))=\pm P_K(\gl)$ (by Prop.~\ref{prop:di-ds}) and the sign depends on $\fb$ (not on $\gl$) by Remark \ref{rem:sign}, this becomes 
\begin{equation*}
  \sum_{\substack{A\subset\gD_n^+(\fl)  \\  C\subset\gD(\fp_1)}} (-1)^{\#A+\#C}P_K(\gl_\gs+\rho_c(\fl)-\rho_n(\fl)+2\rho(A)-2\rho(C)+\rho(\fu)-2\rho(\fp\cap\fu))=c\hspace{.5pt}P_{K\cap L}(\gl_\gs+\rho_c(\fl))
\end{equation*}
up to sign.  Substituting $\gl_\gs$ by $\gl-\rho_c(\fl)-\rho(\fu)-2\rho(\fp\cap\fu)$ this polynomial equation is 
\begin{equation}\label{eqn:poly}
   \sum_{A\subset\gD_n^+,C\subset\gD(\fp_1)} (-1)^{\#A+\#C} P_K(\gl-\rho_n(\fl)+2\rho(A)-2\rho(C))=c\hspace{.5pt}P_{K\cap L}(\gl),
\end{equation}
(since $\rho(\fu), \rho(\fp\cap\fu) \perp\gD_c(\fl)$).

The left-hand side of equation (\ref{eqn:poly}) may be viewed as a sequence of difference operators applied to $P_K(\gl)$.   Define these difference operators by
\begin{equation*}
D_\gb(P(\gl)):=\begin{cases}  P(\gl)-P(\gl+\gb),   &  \gb\in \gD(\fp_1)  \\
                                                 P(\gl+\frac{\gb}{2})-P(\gl-\frac{\gb}{2}),  & \gb\in\gD_n(\fl).
                         \end{cases}
\end{equation*}
These difference operators each lower degree by at least one.  The left-hand side of (\ref{eqn:poly}) is exactly
\begin{equation*}
(-1)^{\#\gD_n(\fl)}\prod_{\gb\in\gD_n(\fl)\cup\gD(\fp_1)}D_\gb(P_K(\gl)),
\end{equation*}
which therefore has degree at most $\deg(P_K(\gl))-\#\gD_n(\fl)-\#\gD(\fp_1)$.  Lemma \ref{lem:degree} below tells us that this degree is at most $\deg(P_{K\cap L}(\gl))=\#\gD_c^+(\fl)$.  However, the left-hand side of (\ref{eqn:poly}) is alternating with respect to the action of $W_{K\cap L}$.   The least degree of such a (nonzero) alternating polynomial is $\deg(P_{K\cap L}(\gl))$.  Therefore, such a polynomial must be a multiple of  $P_{K\cap L}(\gl)$.  Thus there exists some constant $c$ so that (\ref{eqn:poly}) holds.  

The following Lemma now completes the proof that constants $c_k$ as in the theorem exist;  the integrality of the constants holds, but this is to be established by computing the values of the constants; see sections \ref{sec:const-classical} and  \ref{sec:const-exceptional} and \cite{MehdiPandzicVoganZierau17b}. 
\begin{Lem}\label{lem:degree}
 $\#\gD_c^+-\#\gD_c^+(\fl)-\#\gD_n^+(\fl)=\#\gD(\fp_1)$.
\end{Lem}
\begin{proof} We have already noted that $\dim(\cO)=\#\gD^+-\#\gD_c^+$.  The resolution (\ref{eqn:resolution}) gives
\begin{equation}\label{eqn:dimnext}
\begin{split}
\dim(\cO)&=\dim(K/K\cap\bar{Q})+\dim(\fp[2])  \\
&=\dim(\fu\cap \fk)+\dim(\fu\cap \fp)-\dim(\fp_1)  \\
&=\dim(\fu)-\dim(\fp_1).
\end{split}
\end{equation}
Therefore, $\#\gD(\fu)- 
\#
\gD(\fp_1)=\#\gD^+-\#\gD_c^+$, i.e., $\#\gD_c^+-\#\gD_n^+(\fl)-\#\gD(\fp_1)=\#\gD_c^+(\fl).$
\end{proof}

\section{Computations: the classical cases}\label{sec:const-classical}

The computations of the constants are somewhat involved and details will appear in \cite{MehdiPandzicVoganZierau17b}.  Here we first list (in Table 1) all of the classical real groups for which the conjecture applies.  That is all groups with $\rank(G)=\rank(K)$ and $\C[W]\cdot P_K(\gl)=\pi(\cO^\C,1)$, as described in Section \ref{sec:proof}.  We also give the orbit $\cO^\C$ and the number of real forms (i.e., the number of constants).  Then we give one `sample' computation of the constants.

Recall that complex nilpotent orbits in classical Lie algebras are in
one-to-one correspondence with certain partitions  
$\lbrack d_1,\cdots,d_r\rbrack$ with $d_1\geq d_2\geq\cdots\geq
d_r\geq 1$ (if $d_j$ occurs $m$ times, we will simply write $d_j^m$ instead of $[\dots d_j,\dots ,d_j,\dots]$).  The partitions satisfy the following 

\begin{itemize}
\item[$\bullet$] $d_{1}+d_{2}+\cdots+d_{r}=n$, when
  ${\mathfrak g}\simeq{\mathfrak s}{\mathfrak l}(n,{\C})$; 
\item[$\bullet$] $d_{1}+d_{2}+\cdots+d_{r}=2n+1$ and the even $d_j$
  occur with even multiplicity, when ${\mathfrak g}\simeq{\mathfrak s}{\mathfrak o} (2n+1,{\C})$; 
\item[$\bullet$] $d_{1}+d_{2}+\cdots+d_{r}=2n$ and the odd $d_j$ occur
  with even multiplicity, when ${\mathfrak g}\simeq{\mathfrak s}{\mathfrak p}(2n,{\C})$; 
\item[$\bullet$] $d_{1}+d_{2}+\cdots+d_{r}=2n$ and the even $d_j$
  occur with even multiplicity, when ${\mathfrak g}\simeq{\mathfrak s}{\mathfrak o} (2n,{\C})$, except that the partitions having all the $d_j$ even are each associated to two orbits. 
\end{itemize}
See  \cite[Chapter 5]{CollingwoodMcGovern93}.   

It will be convenient to attach to each partition $\lbrack d_1,d_2,\cdots,d_r\rbrack$  a Young tableau, i.e., a left-justified arrangement of empty boxes in rows having $d_i$ boxes in the $i$th row and $d_1$, $d_2$, ..., $d_r$ in non-increasing order.  The real forms will have $\pm$ signs in the boxes.

Let $\gs_K$ denote the $W$-representation generated by the Weyl dimension polynomial $P_K(\gl)$  for $K$.  To check whether $\sigma_K$ is a Springer representation for the classical groups, we proceed
as follows (see \cite[Chapters 11 and 13]{Carter85}): 
\begin{itemize}
\item[(i)] we identify $\sigma_K$ as a Macdonald representation;
\item[(ii)] we compute the symbol of $\sigma_K$ 
as in \cite[Section 13.3]{Carter85};
\item[(iii)] we write down the partition associated with this symbol;
\item[(iv)] we check whether the partition corresponds to a complex
  nilpotent orbit. 
\end{itemize} 
For example, when $G_\R=SU(p,q)$, with $q\geq p\geq 1$, the Weyl group
$W$ is the symmetric group $S_{p+q}$, and $W_K$ can be identified with the subgroup $S_p\times S_q$. The representation
$\sigma_K$ is parametrized, as a Macdonald representation, by the
partition  $\lbrack 2^p,1^{q-p}\rbrack$ (see \cite{Macdonald72} or Proposition 11.4.1 in
\cite{Carter85}). This partition corresponds to a $2pq$-dimensional
nilpotent orbit,  so $\sigma_K$ is Springer. Note that when ${\mathfrak g}$ is 
of type $A_n$, there is no symbol to compute, and any irreducible representation of
$W$  is a Springer representation.   This gives the first entry on our table.

When $G_\R=\Sp (p,q)$, with $q\geq p\geq 1$, the Weyl group $W_K$ is generated by a root subsystem of type $C_p\times C_q$ so that
$\sigma_K$ is parametrized by the pair  
 of partitions $(\lbrack \alpha\rbrack,\lbrack\beta\rbrack)=(\lbrack
 \emptyset\rbrack,\lbrack2^p,1^{q-p}\rbrack)$. The symbol of $\sigma_K$ is the array  
%
\begin{align*}
&\begin{pmatrix}
0& &2& &\cdots& &2(q-p)& &\cdots& &2q\cr
&2& &4&\cdots&2(q-p)& & 2(q-p)+3&\cdots&2q+1&
\end{pmatrix},\text{ if }q>p\text{ and}\\ \\
&\begin{pmatrix}
0&&2&&4&&\cdots & & 2q \cr
&3&&5&&7&\cdots &2q+1&
\end{pmatrix},\text{ if }q=p.
\end{align*}
In both cases, the partition of $2(p+q)$ associated with this symbol does not correspond to a complex nilpotent orbit for ${\mathfrak s}{\mathfrak p}(2(p+q),{\C})$, i.e., $\sigma_K$ is not a Springer representation (see
\cite{Lusztig79} or  \cite[Proposition 11.4.3 and Section 13.
3
]{Carter85}).

\medskip
\begin{table}[H]
\scalebox{1.0}{\mbox{

\setlength{\extrarowheight}{6pt}
\begin{tabular}{|l|c|c|c|}
\hline
  \hspace{15pt}$G_\R$  & $\dim(\Cal O_k)$ &  $\Cal O^\C$ & $\#$ real forms  \\[2pt]
\hline
$SU(p,q)$, $q\geq p\geq 1$ & $2pq$ & $\lbrack 2^p,1^{q-p}\rbrack$ &  $p+1$ \\[2pt]
\hline
 
    &   &  & $3$ if $q>p-1$    \\
\raisebox{12pt}{$SO_{e}(2p,2q+1), q\geq p-1\geq 0$}    & \raisebox{12pt}{$2p(2q+1)$} &  \raisebox{12pt}{$\lbrack 3,2^{2p-2},1^{2(q-p)+2}\rbrack$} & $2$ if $q=p-1$ \\[2pt]
\hline

$\Sp (2n,\R), n\geq 1$  &  $n(n+1)$    &  $\lbrack 2^n\rbrack$    &  $n+1$   \\[2pt]
\hline

 &    &  $\lbrack 2^{n}\rbrack, n$ even &  $\frac{n}{2}+1$\\[2pt]
\raisebox{12pt}{$SO^\star(2n), n\geq 1$}        &               \raisebox{12pt}{$n(n-1)$ }            &  $\lbrack 2^{n-1},1^2\rbrack, n$ odd & $ \frac{n+1}{2}$ \\[2pt]
\hline

       &    &   & $3$ if $q>p$    \\[2pt]
\raisebox{12pt}{$SO_e(2p,2q), q\geq p\geq 1 $}   & \raisebox{12pt}{$4pq$}      &  \raisebox{12pt}{$\lbrack 3,2^{2p-2},1^{2(q-p)+1}\rbrack$}         &  $4$ if $q=p$  \\[2pt]
\hline

\end{tabular}

}}
\caption{}
\end{table}

Recall that our proof of Theorem \ref{thm:main} did not require a choice of $S^\pm$.
However, to determine  numerical values  of the constants we do need to make a choice of $S^\pm$.   This is done  using the positive system $\gD^+$ of the Dynkin diagram as in the discussion surrounding (\ref{eqn:wts}).  This gives a $\gD_c^+$, which in turn gives $P_K(\gl)$, which we fix.  We may choose $\gD^+(\fl)=\gD(\fl)\cap\gD^+$ and set $\fb=\fh+\sum_{\ga\in\gD^+(\fl)} \fg^{(\ga)}+\fu$.  Then  $\fb\subset \fq$ (as required in \S\ref{sec:ext}) and 
\begin{equation}\label{eqn:sign}
DI_p(X_\fb(\gl))=(-1)^NP_K(\gl), \, N=\#\left(\gD^+\smallsetminus \gD(\fb)\right)
\end{equation}
(as in Rem. \ref{rem:sign}).
Observe that $N$ is the number of positive roots that are positive on $h$.  Note also that $\gD^+(\fl)$ give the $\rho_n(\fl)$ in the formula (\ref{eqn:poly}).

As an example of the computation of the constants, we assume $q\geq p$ consider $G_{\R}=SU(p,q)$ and $K_{\R}=S(U(p)\times U(q))$, that is $(G,K)=(SL(n),S(GL(p)\times GL(q))$, $n=p+q$.  We let $\fh$ be the diagonal Cartan subalgebra and write $(a_1,a_2,\cdots,a_{p+q})$ instead of the diagonal matrix $\text{diag}(a_1,a_2,\cdots,a_{p+q})$.  As usual write the roots of $\fh$ in $\fg$ as $\eps_i-\eps_j$, $i\neq j$, where $\eps_i(a_1,\dots,a_n)=a_i$.  The set of roots in $\fk$ is $\gD_c=\{\eps_i-\eps_j: 1\leq i\neq j\leq p\text{ or }p+1\leq i\neq j\leq n\}$.  Let $\gD^+=\{\eps_i-\eps_j:1\leq i<j \leq n\}$.  

The orbit $\cO^\C$ corresponding to the partition $\lbrack 2^p,1^{q-p}\rbrack$ has $p+1$ real forms $\cO_1,\dots,\cO_{p+1}$.  These may be described by signed tableau having $k$ rows, $k=0,1,\dots,p$, beginning with $+$ and $p-k$ rows beginning with $-$ :

\bigskip
$$ \setlength{\extrarowheight}{2pt}\begin{array}{l}
  \left.  \begin{tabular}{|c|c|}
                        \hline
                     +   & $-$\\ [3pt]
                       \hline
                     $\vdots$   &  $\vdots$  \\ [5pt]
                        
                        \hline
                     +   & $-$\\ [3pt]
\end{tabular} \right\}k \\
   \left.   \begin{tabular}{|c|c|}
                       \hline
                     $-$   & +\\ [3pt]
                         \hline
                         
                     $\vdots$    &  $\vdots$  \\ [5pt]
                         \hline
                    $-$   & +\\ [3pt]
                         \hline
\end{tabular}\right\}p-k  \\
  \left.    \begin{tabular}{|c|c|}
                       $-$  \\ [3pt] \cline{1-1} 
                         
                     $\vdots$     \\ [5pt]  \cline{1-1}
                       $-$  \\ [3pt] \cline{1-1} 
                \end{tabular}\right\}q-p
\end{array}
$$
or by the $h$'s of the standard triple:
\begin{equation*}
h_k=(\unb{k}{1,\dots,1},\,\unb{p-k}{-1,\dots,-1}\,|\,\unb{p-k}{1,\dots,1},\,\unb{q-p}{0,\dots,0},\,\unb{k}{-1,\dots,-1}).
\end{equation*}

Since $\fl=\fl_k$ is built from roots that vanish on $h_k$, we have that
\begin{align*}
\Delta^+_n(\fl)=\{\eps_i-&\eps_j\,\big|\, 1\leq i\leq k,\, p+1\leq j\leq 2p-k\} \\ 
&\cup\,\{\eps_i-\eps_j\,\big|\, k+1\leq i\leq p,\, p+q-k+1\leq j\leq p+q\}.
\end{align*}
It follows that for any $A\subseteq\Delta^+_n(\fl)$,
\[
2\rho(A)=(a_1,\dots,a_k,a_1',\dots,a_{p-k}'\,|\,-b_1,\dots,-b_{p-k},0,\dots,0,-b_1',\dots,-b_k'),
\]
with 
\begin{align}
\label{eqn:ineqA}
&0\leq a_i\leq p-k, \quad 0\leq b_j\leq k,\quad \sum_ia_i=\sum_jb_j\quad \text{ and}\\ \nonumber
&0\leq a_i'\leq k, \quad 0\leq b_j'\leq p-k,\quad \sum_ia_i'=\sum_jb_j'.
\end{align}
Furthermore, the set $\Delta(\fp_1)$ of noncompact roots that are 1 on $h_k$ is empty if $q=p$, and for $q>p$,  
\begin{align*}
\Delta(\fp_1)=\{\eps_i&-\eps_j\,\big|\, 1\leq i\leq k,\, 2p-k+1\leq j\leq p+q-k\} \\ 
&\cup\,\{\eps_i-\eps_j\,\big|\, 2p-k+1\leq i\leq p+q-k,\, k+1\leq j\leq p\}.
\end{align*}
It follows that for any $C\subseteq\Delta(\fp_1)$,
\[
2\rho(C)=(c_1,\dots,c_k,-d_1,\dots,-d_{p-k}\,|\,0,\dots,0,-e_1,\dots,-e_{q-p},0,\dots,0),
\]
with 
\begin{equation}
\label{eqn:ineqC}
0\leq c_i,d_j\leq q-p \quad\text{and}\quad -(p-k)\leq e_i\leq k.
\end{equation}

Making the choice of $S^\pm$ described above, equation (\ref{eqn:poly}) becomes
\begin{equation}\label{eqn:C1}
 (-1)^N\sum_{A\subset\gD_n^+,C\subset\gD(\fp_1)} (-1)^{\#A+\#C} P_K(\gl-\rho_n(\fl)+2\rho(A)-2\rho(C))=c\hspace{.5pt}P_{K\cap L}(\gl),
\end{equation}
with $N=kq+(p-k)(q-p+k)+(q-p)k$.  It turns out the computations below are somewhat easier if we make the following manipulation of (\ref{eqn:C1}).  First observe that $-\rho_n(\fl)+2\rho(A)=\rho_n(\fl)-2\rho(\gD_n^+(\fl)\smallsetminus A)$, for any $A\subset\gD_n^+(\fl)$.  Then since $(-1)^{\#A}=(-1)^{\#\gD_n^+(\fl)}(-1)^{\#(\gD_n^+(\fl)\smallsetminus A)}$, equation (\ref{eqn:C1}) becomes
\begin{equation}\label{eqn:C3}
 (-1)^{N+\#\gD_n^+(\fl)}\sum_{A\subset\gD_n^+,C\subset\gD(\fp_1)} (-1)^{\#A+\#C} P_K(\gl-2\rho(A)-2\rho(C))=c\hspace{.5pt}P_{K\cap L}(\gl),
\end{equation}
where we have replaced $\gl$ by $\gl-\rho_n(\fl)$ (noting that $\rho_n(\fl)$ is orthogonal to roots in $\gD_c(\fl)$).

To determine $c$ we evaluate both sides of (\ref{eqn:C3}) at 
\[
\lambda_0=(q,q-1,\dots,q-k+1,p,p-1,\dots,k+1\,|\, p-k,\dots,1,q-k,\dots,p-k+1,k,\dots,1).
\]
Note that  $P_{K\cap L}(\lambda_0)=1$, because $\lambda_0$ differs from $\rho_c(\fl)$ by a weight orthogonal to all roots in $\gD_c(\fl)$.  Set $\Lambda=\lambda_0-2\rho(A)-2\rho(C)$. 

We see that $\Lambda_{p+1},\dots,\Lambda_n$ are $q$ integers lying between $1$ and $q$.   In particular,  note that (\ref{eqn:ineqA}) and (\ref{eqn:ineqC}) imply that $\gL_{2p-k+l}\leq q-l+1, l=1,2,\dots, q-p$  and (\ref{eqn:ineqA}) implies that all  $\gL_{2p+j}, j=1,2,\dots, p-k$ and $\gL_{n-i},i=1,2,\dots, k$ , are at most $p$.   If $P_K(\Lambda)\neq 0$, then  $(\Lambda_{p+1},\dots,\Lambda_n)$ must be a permutation of $(1,\dots,q)$. If $q>p$, then the only possible way to get $p+1,\dots,q-1,q$ is 
\[
\Lambda_{2p-k+1}=q,\ \Lambda_{2p-k+2}=q-1,\ \dots,\ \Lambda_{p+q-k}=p+1.
\]
It follows that $e_1=\dots=e_{q-p}=k$, and hence $c_1=\dots=c_k=q-p$, and $d_1=\dots=d_{p-k}=0$. We conclude that for $q>p$ there is exactly one 
$C\subseteq\Delta(\fp_1)$ for which $P_K(\Lambda)$ can be nonzero: 
\[
C=\{\eps_i-\eps_j\,\big|\, 1\leq i\leq k \text{ and } 2p-k+1\leq j\leq p+q-k\}.
\]
The corresponding $2\rho(C)$ is $(q-p,\dots,q-p,0,\dots,0\,|\, 0,\dots,0,-k,\dots,-k,0,\dots,0)$, and $\#C=k(q-p)$.
This also covers the case $q=p$; in that case, 
$\Delta(\fp_1)=\emptyset$, so $C=\emptyset$ and $2\rho(C)=0$. It follows that 
\begin{equation*}\begin{split}
\Lambda&=
({p-a_1},{p-1-a_2},\dots,{p-k+1-a_k},\, p-a_1',{p-1-a_2'},\dots,{k+1-a_{p-k}'}\,|\, \\
&{p-k+b_1},{p-k-1+b_2},\dots,{1+b_{p-k}},\,q,\dots,p+1,\,{k+b_1'},{k-1+b_2'},\dots,{1+b_k'}).
\end{split}\end{equation*}

Note that $P_K(\gL)\neq 0$ if and only if $\{\gL_1,\dots,\gL_p\}$ and $\{\gL_{p+1},\dots,\gL_{p-k},\gL_{n-k+1},\dots,\gL_{n}\}$ are each a permutation of $\{1,2,\dots,p\}$.  To determine when this happens we may assume that $p=q$. Therefore we consider 
\begin{equation*}\begin{split}
\gL&=\gl_0-2\rho(A)  \\
&=({p-a_1},{p-1-a_2},\dots,{p-k+1-a_k},\, p-a_1',{p-1-a_2'},\dots,{k+1-a_{p-k}'}\,|\, \\
&\quad\quad {p-k+b_1},{p-k-1+b_2},\dots,{1+b_{p-k}},\,{k+b_1'},{k-1+b_2'},\dots,{1+b_k'}).
\end{split}\end{equation*}
By (\ref{eqn:ineqA}) and (\ref{eqn:ineqC}) the coordinates of $\gL$ satisfy the following inequalities:
\begin{equation}\label{eqn:ineq}
\begin{array}{rll}
k-i+1\leq &\hspace{-5pt}\gL_i  \leq p-i+1, & i=1,\dots,k \\
p-k-j+1\leq &\hspace{-5pt} \gL_{k+j}  \leq p-j+1, & j=1,\dots,p-k \\
p-k-j+1\leq &\hspace{-5pt} \gL_{p+j}  \leq p-j+1, & j=1,\dots,p-k \\
k-i+1\leq &\hspace{-5pt} \gL_{2p-k+i}  \leq p-i+1 , & i=1,\dots,k. 
\end{array}
\end{equation}

\noindent \underline{\bf{Claim:}} Under the assumption that $p=q$ and $\gL=\gl_0-2\rho(A)$ as above, $P_K(\gL)\neq 0$ if and only if there is a shuffle $i_1>\cdots>i_k, j_1>\cdots j_{p-k}$ of $p,p-1,\dots 2,1$ so that 
\begin{align*}
A=&\{\ga_{a,b}: 1\leq a\leq k, 1\leq b\leq p-k\}  \\
&\ga_{a,b}=\begin{cases} \eps_a-\eps_{p+b},  &\text{if }i_a<j_b  \\
                         \eps_{k+b}-\eps_{2p-k+a}, &\text{if }i_a>j_b.
           \end{cases}
\end{align*}
Observe that for such an $A$, 
\begin{equation}\label{eqn:shuffle}
\gL=(i_1,\dots, i_k, j_1,\dots , j_{p-k} \,|\,  j_1,\dots , j_{p-k},i_1,\dots, i_k)
\end{equation}
and $P_K(\gL)\neq 0$, and $\#A=k(p-k)$.

We prove that if $P_K(\gL)\neq 0$ then there is such a shuffle by induction on $p$.  If $k=0$ or $k=p$ this is immediate, since $\gD_n(\fl)$ and $A$ are empty.  In particular the claim holds if $p=1$.

Assume that $p>1$ and $1\leq k\leq p-1$.  Since $P_K(\gL)\neq 0$, the inequalities (\ref{eqn:ineq}) require that $\gL_1=p$ or $\gL_{k+1}=p$.

\noindent {\bf Case 1:} $\gL_1=p$, i.e., $a_1=0$.  This means that $\eps_1-\eps_{p+b}\notin A$, for $b=1,2,\dots , p-k.$  It follows that for $j=1,2,\dots,p-k$,
$$
b_j\leq k-1 \quad\text{and}\quad  p-k\leq \gL_{p+j} \leq p-j.
$$
Therefore $\gL_{2p-k+1}=p$ (since $\gL_{p-1}\leq p-1$).  This means that $b_1'=p-k$ and $a_j'\geq 1, j=1,2,\dots,p-k$.  So $\eps_{k+b}-\eps_{2p-k+1}\in A$, for $b=1,2,\dots,p-k$.  Ignoring the coordinates $\gL_1=p$ and $\gL_{2p-k+1}=p$ we are left with
\begin{equation}\label{eqn:C5}
\begin{split}
\gL&= (\bullet, p-1-a_2,\dots, p-k+1-a_k,  p-1-a_1'',\dots, k-a_{p-k}'' \,|\,  \\
   & p-k-+b_1,\dots, 1+b_{p-k}, \bullet, k-1+b_2',\dots, 1+b_k')
\end{split}
\end{equation}
where $0\leq a_i''=a_i-1 \leq k-1$ and the first $p-1$ coordinates are a permutation of $\{1,2,\dots,p-1\}$, as are the last $p-1$ coordinates. We are precisely in the $p-1$ case with $k$ replaced by $k-1$.  Induction says that (\ref{eqn:C5}) is 
\begin{equation*}
\gL=(p,i_2,\dots, i_k, j_1,\dots , j_{p-k} \,|\,  j_1,\dots , j_{p-k},p,i_2,\dots, i_k)
\end{equation*}
for some shuffle of $p-1,\dots 2,1$.  This gives an $\tilde{A}$ as in the claim (for $p-1$), which along with  $\eps_{k+b}-\eps_{2p-k+1}$, for $b=1,2,\dots,p-k,$ corresponds to the shuffle   $p=i_1>\cdots>i_k, j_1>\cdots j_{p-k}$ of $p,\dots 2,1$.

\noindent {\bf Case 2:} $\gL_{k+1}=p$, i.e., $a_1'=0$. This case is essentially the same as Case 1. We see that $\gL_{p+1}=p$ and $\eps_a-ge_{p+1}\in A, a=1,2,\dots,k$.  A reduction to $p-1$ and the same $k$ gives $\gL=\gl_0-2\rho(A)$, with $A$ defined by a shuffle $i_1>\cdots>i_k, p=j_1>\cdots j_{p-k}$

Thus, $\gL$ arises from a shuffle when $P_K(\gL)\neq 0$.
Note that the two cases result in all possible shuffles.  The claim is proved.

To compute $c_k^{\hspace{1pt}p,q}$, it remains to compute $P_K(\Lambda)$ for the $\Lambda$ as in (\ref{eqn:shuffle}).  We have that $P_K(\Lambda)=\pm 1$ for each such $\Lambda$, since $P_K(p,\dots,2,1\,|\, q,\dots,2,1)=1$.  To compute the sign, we first use $(p-k)(q-p+k)$ transpositions to bring $\Lambda$ to 
\[
\gL=(i_1,\dots,i_k,j_1,\dots,j_{p-k}\,|\, q,\dots,p+1, i_1,\dots,i_k, j_1,\dots,j_{p-k}).
\]
Now the number of transpositions needed to bring the coordinates to the left of the bar to descending order is equal to the number of transpositions needed to bring the coordinates to the right of the bar to descending order.
So the total number of transpositions we need at this step is even and so for each $\Lambda$ as in \ref{eqn:shuffle}
\[
P_K(\Lambda)=(-1)^{(p-k)(q-p+k)}.
\]
As we noted above, each relevant $A$ has $k(p-k)$ elements, and the only relevant set $C$ has $k(q-p)$ elements. The number of positive roots that are positive on $h_k$ is
$N=kq+(p-k)(q-p+k)+(q-p)k.$
Finally, the number of possible $\Lambda$ is the number of $(k,p-k)$-shuffles, i.e., $\binom{p}{k}$. So
\[
c_k^{\hspace{1pt}p,q}=(-1)^{k(n-k)}\binom{p}{k}
\]
Therefore we have computed the constant for the $k^{\rm th}$ real form.

\section{Computation of the constants: the exceptional cases}\label{sec:const-exceptional}

The following table lists all of the pairs $(G,K)$, with $G$ an exceptional simple group, for which the $W$ representation $\C[W]\cdot P_K(\gl)$  is Springer.  The table gives the corresponding complex orbits $\cO^\C$ in notation of \cite[\S13.3]{Carter85}.  Also given are the number of real forms and the values of the constants in each case.

\medskip
\begin{table}[H]
\setlength{\extrarowheight}{6pt}
\begin{tabular}{|l|c|c|c|c|c|c|}
\hline
  $G$  & $K$ & $\dim(\Cal O_k)$ &  $\Cal O^\C$ &  $\#$ real forms & constants\\
\hline
E6  & $A5\times A1$    & 20     & $3A_1$    & 2   &  4, 12   \\
      & $D5\times \C$     & 16     & $2A_1$    &  3  &  1, -2, 1    \\[2pt]
\hline
E7  & $D6\times A1$    & 32       &  $3A_1'$   &  2   &  4,  12,   \\
      & $A7$                    & 35       &  $4A_1$    &    2 &   64,  -64  \\
      & $E6\times \C$     & 27        & $3A_1''$   &    4  &   -1, 3, -3, 1  \\[2pt]
\hline
E8  & $E7\times A1\,^*$ & 56    &  $3A_1$    &   2   & 4, 12    \\
      & $D8$                      & 64    &  $4A_1$       &  1   &   256   \\[2pt]
\hline
F4  & $C3\times A1$      & 14    & $A_1+\widetilde{A_1}$ & 2  &   4, -4    \\[2pt]
\hline
G2  & $A1\times A1$     & 4      &  $\widetilde{A_1}$          &  1 &    4    \\[2pt]
\hline
\end{tabular}
\setlength{\extrarowheight}{0pt}\\
\caption{}
\end{table}

The case of $(F4, \Spin (9))$ is the only pair with $G$  exceptional (and $K$ of equal rank) for which  $\C[W]\cdot P_K(\gl)$ is not Springer.  

The pairs appearing in the table and the complex orbits $\cO^\C$ may be determined as follows.  The representation $\C[W]\cdot P_K(\gl)$ occurs in homogeneous degree $d=\#\gD_c^+$ polynomials on $\fh^*$.  It occurs in no lower degree since it contains the sign representation of $W_K$, which occurs in no lower degree.  The tables in \cite[\S13.3]{Carter85} show that $\C[W]\cdot P_K(\gl)$ must be of the form $\phi_{m,d}$ for some $m$.  In most cases there is just one possibility and the table tells us this one possibility is Springer.  The case of $(F4,\Spin (9))$ has two possibilities $\phi_{2,16}'$ and $\phi_{2,16}''$.  Only $\phi_{2,16}'$ contains the sign representation of $W_K$ (as given in \cite{Alvis05}).  However, $\phi_{2,16}'$ is not Springer (by \cite{Carter85}).

The real forms are listed below by giving the associated labelled (extended) Dynkin diagram as in \cite[Ch.~9]{CollingwoodMcGovern93}.  Given a nilpotent orbit there is a normal triple $\{x,h,y\}$ and a Dynkin diagram labelled with $\ga(h)$ for each compact simple root.  It may be arranged that each $\ga(h)\geq 0$.  In the diagrams below, $\gg$ is the highest root and the noncompact roots are blackened.  Recall that the $h$ determines $\fq$.  The choice of $S^\pm$ is made as described in \ref{eqn:sign}.

The constants are computed by evaluating both sides of equation (\ref{eqn:poly}) at a convenient point (such as $\gl=\rho_c(\fl)$) by using a computer\footnote{The computations for $E8$ are particularly long and were performed  at the OSU High Performance Computing Center at Oklahoma State University,  supported in part through the National Science Foundation grant OCI-1126330.  They were also carried out on the computer cluster of the Institut Elie Cartan de Lorraine (UMR 7502-CNRS).}.  On Table 2 the constants correspond to the real forms below (and are listed in the same order as the real forms are listed).




\noindent (1)  $E6,\, \f{e}_{6(2)},\,K=A5\times A1$.   Two real forms.
\vspace{41pt}

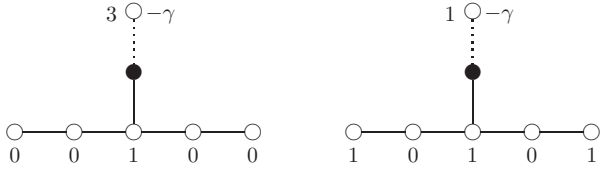
\scalebox{.75}{\mbox{\begin{picture}(0,0)
\multiput(5,0)(30,0){5}{\circle{8}}
\multiput(9,0)(30,0){4}{{\line(1,0){22}}}
\multiputlist(5,-12)(30,0){$0$,$0$,$1$,$0$,$0$}
\put(65,30){\circle*{8}}\put(65,4){\line(0,1){22}}
\put(65,61){\circle{8}}\put(51,56){$3$}\put(71,57){$-\gg$}
\thicklines{\dottedline{4}(65,34)(65,56)}
\end{picture}
}
\hspace{160pt}
\mbox{\begin{picture}(0,0)
\multiput(5,0)(30,0){5}{\circle{8}}
\multiput(9,0)(30,0){4}{{\line(1,0){22}}}
\multiputlist(5,-12)(30,0){$1$,$0$,$1$,$0$,$1$}
\put(65,30){\circle*{8}}\put(65,4){\line(0,1){22}}
\put(65,61){\circle{8}}\put(51,56){$1$}\put(71,57){$-\gg$}
\thicklines{\dottedline{4}(65,34)(65,56)}
\end{picture}
}} 
\vspace{14pt}

\noindent (2) $E6,\, \f{e}_{6(-14)},\,K=D5\times \C$.  Three real forms.
\vspace{26pt}

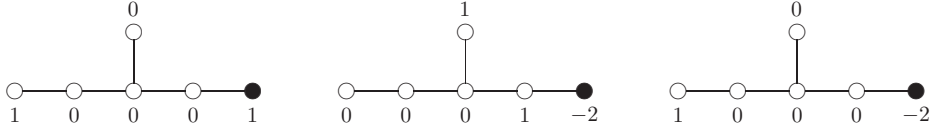
\scalebox{.75}{\mbox{\begin{picture}(0,0)
\multiput(5,0)(30,0){5}{\circle{8}}
\multiput(9,0)(30,0){4}
{{\line(1,0){22}}}
\multiputlist(5,-12)(30,0){$1$,$0$,$0$,$0$,$1$}
\put(65,30){\circle{8}}\put(65,4){\line(0,1){22}}\put(62,38){$0$}
\put(125,0){\circle*{8}}
\end{picture}}
\hspace{160pt}
\mbox{\begin{picture}(0,0)
\multiput(5,0)(30,0){5}{\circle{8}}
\multiput(9,0)(30,0){4}
{{\line(1,0){22}}}
\multiputlist(5,-12)(30,0){$0$,$0$,$0$,$1$,$-2$}
\put(65,30){\circle{8}}\put(65,4){\line(0,1){22}}\put(62,38){$1$}
\put(125,0){\circle*{8}}
\end{picture}}
\hspace{160pt}
\mbox{\begin{picture}(0,0)
\multiput(5,0)(30,0){5}{\circle{8}}
\multiput(9,0)(30,0){4}
{{\line(1,0){22}}}
\multiputlist(5,-12)(30,0){$1$,$0$,$0$,$0$,$-2$}
\put(65,30){\circle{8}}\put(65,4){\line(0,1){22}}\put(62,38){$0$}
\put(125,0){\circle*{8}}
\end{picture}}
}
\vspace{14pt}

\noindent (3) $E7,\, \f{e}_{7(-5)},\, K=D6\times A1$.   Two real forms.
\vspace{23pt}

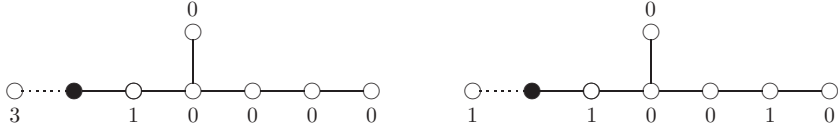
\scalebox{.75}{\mbox{\begin{picture}(0,0)
\multiput(35,0)(30,0){6}{\circle{8}}\put(65,0){\circle{8}}
\multiput(39,0)(30,0){5}
{{\line(1,0){22}}}
\multiputlist(5,-12)(30,0){$3$,,$1$,$0$,$0$,$0$,$0$}
\put(95,30){\circle{8}}\put(95,4){\line(0,1){22}}\put(92,38){$0$}
\put(5,0){\circle{8}}\put(5,-12){}\thicklines{\dottedline{4}(9,0)(31,0)}
\put(35,0){\circle*{8}}
\end{picture}
}
\hspace{220pt}
\begin{picture}(0,0)
\multiput(35,0)(30,0){6}{\circle{8}}\put(65,0){\circle{8}}
\multiput(39,0)(30,0){5}
{{\line(1,0){22}}}
\multiputlist(5,-12)(30,0){$1$,,$1$,$0$,$0$,$1$,$0$}
\put(95,30){\circle{8}}\put(95,4){\line(0,1){22}}\put(92,38){$0$}
\put(5,0){\circle{8}}\put(5,-12){}\thicklines{\dottedline{4}(9,0)(31,0)}
\put(35,0){\circle*{8}}
\end{picture}
}
\vspace{14pt}

\noindent (4) $E7,\, \f{e}_{7(7)},\,  K=A7$.  Two real forms.

\vspace{17pt}

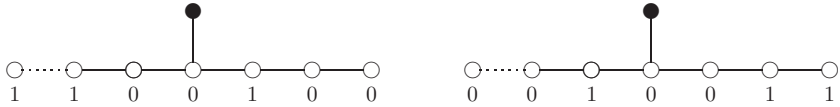
\scalebox{.75}{\mbox{\begin{picture}(0,0)
\multiput(35,0)(30,0){6}{\circle{8}}\put(65,0){\circle{8}}
\multiput(39,0)(30,0){5}
{{\line(1,0){22}}}
\multiputlist(5,-12)(30,0){$1$,$1$,$0$,$0$,$1$,$0$,$0$}
\put(95,30){\circle*{8}}\put(95,4){\line(0,1){22}}
\put(5,0){\circle{8}}\thicklines{\dottedline{4}(9,0)(31,0)}
\end{picture}
}
\hspace{220pt}
\begin{picture}(0,0)
\multiput(35,0)(30,0){6}{\circle{8}}\put(65,0){\circle{8}}
\multiput(39,0)(30,0){5}
{{\line(1,0){22}}}
\multiputlist(5,-12)(30,0){$0$,$0$,$1$,$0$,$0$,$1$,$1$}
\put(95,30){\circle*{8}}\put(95,4){\line(0,1){22}}
\put(5,0){\circle{8}}\thicklines{\dottedline{4}(9,0)(31,0)}
\end{picture}
}
\vspace{14pt}

\noindent (5) $E7,\, \f{e}_{7(-25)},\, K=E6\times \C$.  Four real forms.

\vspace{17pt}

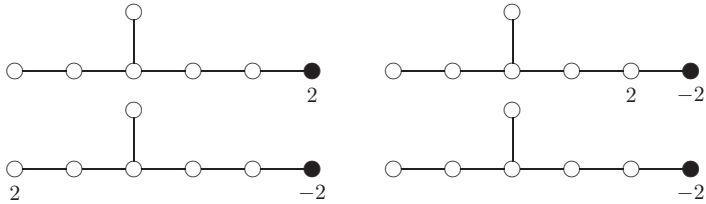
\scalebox{.75}{\mbox{\begin{picture}(0,0)
\multiput(5,0)(30,0){6}{\circle{8}}\put(155,0){\circle*{8}}
\multiput(9,0)(30,0){5}
{{\line(1,0){22}}}
\multiputlist(5,-12)(30,0){,,,,,$2$}
\put(65,30){\circle{8}}\put(65,4){\line(0,1){22}}
\end{picture}
}
\hspace{180pt}
\begin{picture}(0,0)
\multiput(5,0)(30,0){6}{\circle{8}}\put(155,0){\circle*{8}}
\multiput(9,0)(30,0){5}
{{\line(1,0){22}}}
\multiputlist(5,-12)(30,0){,,,,$2$,$-2$}
\put(65,30){\circle{8}}\put(65,4){\line(0,1){22}}
\end{picture}
}
\vspace{19pt}

\scalebox{.75}{\mbox{\begin{picture}(0,0)
\multiput(5,0)(30,0){6}{\circle{8}}\put(155,0){\circle*{8}}
\multiput(9,0)(30,0){5}
{{\line(1,0){22}}}
\multiputlist(5,-12)(30,0){$2$,,,,,$-2$}
\put(65,30){\circle{8}}\put(65,4){\line(0,1){22}}
\end{picture}
}
\hspace{180pt}
\begin{picture}(0,0)
\multiput(5,0)(30,0){6}{\circle{8}}\put(155,0){\circle*{8}}
\multiput(9,0)(30,0){5}
{{\line(1,0){22}}}
\multiputlist(5,-12)(30,0){,,,,,$-2$}
\put(65,30){\circle{8}}\put(65,4){\line(0,1){22}}
\end{picture}
}
\vspace{14pt}

\noindent (6) $E8,\,\f{e}_{8(-24)}, K=E7\times A1$.    Two real forms.
\vspace{19pt}

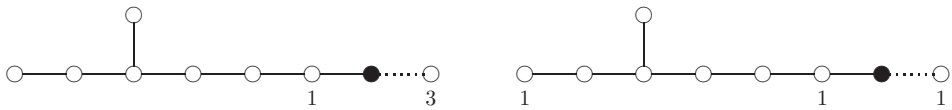
\scalebox{.75}{\mbox{\begin{picture}(0,0)
\multiput(5,0)(30,0){7}{\circle{8}}
\multiput(9,0)(30,0){6}{{\line(1,0){22}}}
\multiputlist(5,-12)(30,0){,,,,,$1$,,$3$}
\put(65,30){\circle{8}}\put(65,4){\line(0,1){22}}
\put(215,0){\circle{8}}\thicklines{\dottedline{4}(189,0)(211,0)}
\put(185,0){\circle*{8}}
\end{picture}
\hspace{250pt}
{\begin{picture}(0,0)
\multiput(5,0)(30,0){7}{\circle{8}}
\multiput(9,0)(30,0){6}{{\line(1,0){22}}}
\multiputlist(5,-12)(30,0){$1$,,,,,$1$,,$1$}
\put(65,30){\circle{8}}\put(65,4){\line(0,1){22}}
\put(215,0){\circle{8}}\thicklines{\dottedline{4}(189,0)(211,0)}
\put(185,0){\circle*{8}}
\end{picture}
}} }
\vspace{14pt}

\noindent (7)  $E8,\,\f{e}_{8(8)}, K=D8$.  Just one real form.
\vspace{17pt}

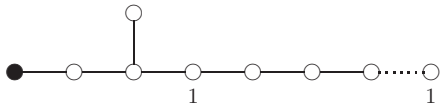
\scalebox{.75}{\mbox{\begin{picture}(240,-50)
\multiput(5,0)(30,0){7}{\circle{8}}
\multiput(9,0)(30,0){6}{{\line(1,0){22}}}
\multiputlist(5,-12)(30,0){,,,$1$,,,,$1$}
\put(65,30){\circle{8}}\put(65,4){\line(0,1){22}}
\put(215,0){\circle{8}}\thicklines{\dottedline{4}(189,0)(211,0)}
\put(5,0){\circle*{8}}
\end{picture}
}} 

\vspace{14pt}

\noindent (8) $F4,\,\f{f}_{4(4)},\; K=C3\times A1$.  Two real forms.

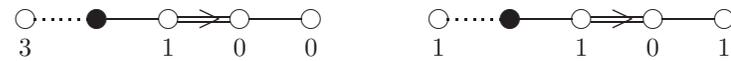
\scalebox{.9}{\mbox{
\begin{picture}(0,0)
\multiput(35,0)(30,0){4}{\circle{8}}
\multiput(39,0)(60,0){2}
{{\line(1,0){22}}}
\multiputlist(5,-12)(30,0){$3$,,$1$,$0$,$0$}
\put(69,1){\line(1,0){22}}\put(69,-1){\line(1,0){22}}\put(77,-3){\scalebox{1.25}{\mbox{$>$}}}
\put(35,0){\circle*{8}}
\put(5,0){\circle{8}} \thicklines\dottedline{4}(9,0)(31,0)
\end{picture}\hspace{170pt}
\begin{picture}(0,0)
\multiput(35,0)(30,0){4}{\circle{8}}
\multiput(39,0)(60,0){2}
{{\line(1,0){22}}}
\multiputlist(5,-12)(30,0){$1$,,$1$,$0$,$1$}
\put(69,1){\line(1,0){22}}\put(69,-1){\line(1,0){22}}\put(77,-3){\scalebox{1.25}{\mbox{$>$}}}
\put(35,0){\circle*{8}}
\put(5,0){\circle{8}} 
\thicklines\dottedline{4}(9,0)(31,0)
\end{picture}
}}

\vspace{14pt}

\noindent (9) $G2, K=A1\times A1$.  Just one real form.

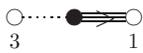
\scalebox{.75}{\mbox{\begin{picture}(0,0)
\multiput(5,0)(30,0){3}{\circle{8}}
\multiput(39,-2)(0,2){3}{{\line(1,0){22}}}
\multiputlist(5,-12)(30,0){$3$,,$1$}
\put(45,-4.3){\scalebox{1.25}{\mbox{\larger$>$}}}
\put(35,0){\circle*{8}}
\thicklines\dottedline{4}(9,0)(31,0)
\end{picture}
}}
\vspace{14pt}

\appendix
\section{}
Consider a variety $X$ and an algebraic group $H$ acting on $X$ with a finite number of orbits.  
 
As mentioned in \S\ref{sec:AC}, if $\cO=H\cdot x$, then 
\begin{equation}\label{eqn:homog}
\textup{K}^H(X)\simeq \Rep(H).
\end{equation}
If $i:Z\hookrightarrow X$ is a closed embedding of an $H$-stable subset $Z$ of $X$ and $U:=X\smallsetminus Z$, with $j:U\to X$ the inclusion map, then there is an exact sequence
 \begin{equation}\label{eqn:seq-a}
 \cdots\longrightarrow \textup{K}_1^H(U)\longrightarrow  \textup{K}^H(Z) \overset{i_*}{\longrightarrow } \textup{K}^H(X) \overset{j^*}{\longrightarrow } \textup{K}^H(U) \longrightarrow  0.
 \end{equation}
 
 For any homogeneous vector bundle $H\underset{H^x}{\times} E_\tau$ on $\cO=H\cdot x$, $\cE_\cO(\tau)$ is the sheaf of local sections.  Then by surjectivity of $j^*$ in the exact the sequence (\ref{eqn:seq-a}) applied to $X=\bar{\cO}, Z=\bar{\cO}\smallsetminus \cO$, $\cE_\cO(\tau)$ extends to the closure of $\cO$.   Extending further (by zero) to a coherent sheaf on all of $X$ gives an $H$-equivariant coherent sheaf $\tilde{\cE}_{\cO}(\tau)$ on $X$ extending $\cE_\cO(\tau)$.  
 
For the action of $H$ on $X$, list the orbits as $\cO_1,\dots,\cO_m$  in a way that is compatible with the closure relations, i.e., so that $\cO_k\subset\bar{\cO}_j$ implies $k\leq j$.  Choose base points $x_k$ in each orbit and write $\cO_k= K\cdot x_k$.  Applying the construction above to each orbit and each isotropy representation we obtain coherent sheaves $\tilde{\cE}_{\cO_k}(\tau_{kl})$.
 
\begin{Lem}\label{lem:span}  $\textup{K}^H(X)$ is spanned (over $\Z$) by $\{\tilde{\cE}_{\cO_k}(\tau_{kl}) :k=1,\dots,m, \tau_{kl}\in (H^{x_k})^{\widehat{\,}}\;\}$.
\end{Lem}

\begin{proof} We use induction on the number $m$ of orbits in $X$.  In our listing of orbits $\cO_m$ will be an open orbit.  Let $U=\cO_m$ and $Z=X\smallsetminus U$.   If the restriction of the equivariant sheaf $\cF$ to $U$ is zero, then $[\cF]$ is in the image of $i_*$.  By induction each element of $\textup{K}^H(Z)$ is an integer combination of the $\tilde{\cE}_{\cO_k}'(\tau_{kl}), k=1,\dots,m-1$ (where $\tilde{\cE}_{\cO_k}'(\tau_{kl})$ means the extensions to $Z$, instead of $X$, so $\tilde{\cE}_{\cO_k}(\tau_{kl})=i_*(\tilde{\cE}_{\cO_k}'(\tau_{kl})$)).  Therefore, $[\cF]$ is an integer combination of the $\tilde{\cE}_{\cO_k}(\tau_{kl}), k=1,\dots, m-1.$  If the restriction of $\cF$ to $U$ is not zero, then write this restriction as an integer combination  $\sum_{j}n_{mj}\cE_{\cO_m}(\tau_{mj})$,  Now $[\cF]-\sum_{j}n_{mj}[\tilde{\cE}_{\cO_m}(\tau_{mj})]$ is in the image of $i_*$, so (as above)  is an integer combination of the $\tilde{\cE}_{\cO_k}(\tau_{kl}), k=1,\dots, m-1.$ 
\end{proof}
 
\begin{Thm}\label{thm:basis-a} Suppose the algebraic group $H$ acts on a variety $X$ with $m$ of orbits and one constructs $\cB:=\{[\tilde{\cE}_{\cO_k}(\tau_{kl})] :k=1,\dots,m, \tau_{kl}\in (H^{x_k})^{\widehat{\,}}\;\}\subset \textup{K}^H(X)$ as above.  Then
\vspace{-8pt}
\begin{enumerate} 
\item for any closed embedding $i:Z\to X$ of an $H$-stable subset $Z$ of $X$,  $i_*: \textup{K}^H(Z)\to \textup{K}^H(X)$ is injective and 
\item $\cB$ is a $\Z$-basis of $\textup{K}^H(X)$.
\end{enumerate}
\end{Thm}

We prove this by induction on the number of orbits of $H$ on $X$.  If there is just one orbit, then $X=\cO=H\cdot x$  and the coherent sheaves $\cE_\cO(\tau)$ are a basis; this is (\ref{eqn:homog}).   (1) holds trivially.
 
Now suppose the theorem holds for all $H$-spaces with fewer than $m$ orbits.  List the orbits  as $\cO_1,\dots,\cO_m$ compatible with the closure relations.

We prove (1) under the inductive hypothesis.

 \noindent{\bf Claim:}  If $\cO\subset X$ is any open $H$-orbit dense in $X$ and $Z=X\smallsetminus\cO$,  then the  map 
\begin{equation}\label{eqn:star}\tag{*}
i_*: \textup{K}^H(Z)\to  \textup{K}^H(X)
\end{equation}
is injective.

Given the exact sequence (\ref{eqn:seq-a}), it will suffice to show that there is no nonzero homomorphism $\textup{K}_1^H(\cO)\to  \textup{K}^H(Z)$.  We compute $\textup{K}_1^H(\cO)$.  This is $\textup{K}_1$ of the category of $H$-equivariant sheaves on the orbit, that is, the category of (algebraic) representations of $H^x$, where $\cO=H/H^x$.  We will see that 
$$\textup{K}_1^H(\cO)\simeq  \oplus \C^\times,$$
one summand for each irreducible representation of $H^x$.  Once this is done, the claim will be proved, since our induction hypothesis says $\textup{K}^H(Z)$ is a free abelian group and there are no nonzero homomorphisms from a divisible group into a free abelian group.

For the computation of $\textup{K}_1^H(\cO)$, we will follow the discussion in \cite[Ch. V, \S4]{Weibel13}.  $\textup{K}_1^H(\cO)$ is $\textup{K}_1$ of the category of $H^x$-representations.  If we replace this category by its (exact) subcategory of semisimple $H^x$-representations, then the $\textup{K}$-theory remains the same; this is Quillen's `D\'evissage' (\cite[Thm. 4]{Quillen73} or \cite[Thm. 4.1]{Weibel13}).  By \cite[4.3]{Weibel13}, we get $\textup{K}_1^H(\cO)\simeq \oplus \textup{K}_1(\C)$, with one summand for each irreducible $H^x$-representation.  But it is well-known that $\textup{K}_1(\C)\simeq \C^\times$ (e.g., \cite[Ch. III, \S1]{Weibel13}).  We have now shown that $\displaystyle\textup{K}_1^H(\cO)\simeq \bigoplus_{\tau\in (H^x)^{\dual}} \C^\times,$ completing the proof of the claim.

Now for the proof of (1).   We may assume that $\cO_m\subset X\smallsetminus Z$ (by replacing $\cO_m$ by an open orbit in the open set $X\smallsetminus Z$, if necessary).  Set $Y=X\smallsetminus\cO_m$.  Then we have closed embeddings $i_1:Z\hookrightarrow Y$ and $i_2:Y\hookrightarrow X$, the composition being $i$.  Therefore $i_*={i_2}_*\, {i_1}_*$ is injective, since $i_2^*$ is an injection by the claim and $i_1^*$ is an injection by induction.
 
Now we show part (2) by verifying that $\cB$ is independent.

Suppose
 \begin{equation}\label{eqn:combo}
 \sum_{k=1}^m\sum_l n_{kl}\tilde{\cE}_{\cO_k}(\tau_{kl})=0 \text{ in } \textup{K}^H(X).
 \end{equation}
Let $Z=X\smallsetminus \cO_m$ (recall that $\cO_m$ is open) and let
$$i:Z\hookrightarrow X  \text{ and }j:\cO_m\hookrightarrow X.$$
Apply $j^*$ to both sides of (\ref{eqn:combo}) and get 
$$\sum_l n_{ml}{\cE}_{\cO_m}(\tau_{kl})=0, \text{ in } \textup{K}^H(\cO_m).$$
This implies $n_{ml}=0$, all $l$ (by (\ref{eqn:homog})).  By exactness of (\ref{eqn:seq-a})
$$\sum_{k=1}^{m-1}\sum_l n_{kl}\tilde{\cE}_{\cO_k}(\tau_{kl})$$
 lies in the image of $i_*$.  For the moment, write ${\cE}_{\cO_k}'(\tau_{kl})$ for an extension of ${\cE}_{\cO_k}(\tau_{kl})$ to $\bar{\cO}_k$, and $\tilde{\cE}_{\cO_k}'(\tau_{kl})$ for its  extension (by zero) to $Z$.  So $i_*(\tilde{\cE}_{\cO_k}'(\tau_{kl}))=\tilde{\cE}_{\cO_k}(\tau_{kl})$.  By the injectivity of $i_*$ we have
 $$ \sum_{k=1}^{m-1}\sum_l n_{kl}\tilde{\cE}_{\cO_k}'(\tau_{kl})=0, \text{ in } \textup{K}^H(Z).$$
 By induction, $n_{kl}=0,k=1,\dots,m-1, $ all $l$.  Thus, independence is proved, completing the induction argument.
  \hfill{$\square$}
 
\begin{Rem} Part (1) of the theorem can fail  without the assumption that there are a finite number of $H$-orbits on $X$.  
Consider the example of $H=\{e\}$ with trivial action on $X=\C$; $Z=\{0\}$ is the complement of the open set $\C^\times\subset X$.   There is an exact sequence
$$0\to \C[x] \overset{x}{\to} \C[x] \to \C[x]/(x)\to 0.$$
So $[ \C[x]/(x)]=0$ in $ K^{\{e\}}(\C)$, this is the extension by zero of the constant sheaf on $\{0\}$.
Note that the map `multiplication' by $x$ is not $\C^\times$-equivariant, so the same argument does not apply in $\C^\times$-equivariant $\textup{K}$-theory.  
\end{Rem}

\begin{Cor}\label{cor:well-def}  Suppose $i_1:Z_1\hookrightarrow X$ and $i_2:Z_2\hookrightarrow X$ are closed embeddings of $H$-stable subsets and $i_{12}:Z_1\cap Z_2\hookrightarrow X$.  If $\cE\in \textup{K}^H(X)$ is in the image of both ${i_1}_*$ and ${i_2}_*$, then $\cE$ is in the image of ${i_{12}}_*: \textup{K}^H(Z_1\cap Z_2)\to \textup{K}^H(X)$.
\end{Cor}
\begin{proof} For $p=1,2$ write $\cE={i_p}_*(\cF_p)$ with $i_p:Z_p\hookrightarrow X$ and $\cF_p\in \textup{K}^H(Z_p)$.  We know that extensions of the vector bundles on $\cO_k\subset Z_p$ form a basis $\{[\tilde{\cE}_{\cO_k}^p(\tau_{kl})]\}$ of $ \textup{K}^H(Z_p)$.  Expressing $\cF_p$ in terms of this basis gives 
$$\cF_p=\sum_{\cO_k\subset Z_p}\sum_l n_{kl}^p [\tilde{\cE}_{\cO_k}^p(\tau_{kl})], \,p=1,2.
$$
Noting that ${i_p}_*([\tilde{\cE}_{\cO_p}^k(\tau_{kl})])=[\tilde{\cE}_{\cO_k}(\tau_{kl})]$, when $\cO_k\subset Z_p$, gives
$$\cE=\sum_{\cO_k\subset Z_1}\sum_l n_{kl}^1 [\tilde{\cE}_{\cO_k}^1(\tau_{kl})]=\sum_{\cO_k\subset Z_2}\sum_l n_{kl}^2 [\tilde{\cE}_{\cO_k}^2(\tau_{kl})].
$$
By independence, it follows that only orbits $\cO_k$ contained in both $Z_1$ and $Z_2$ appear with nonzero coefficients.  That is,
$$\cE=\sum_{\cO_k\subset Z_1\cap Z_2}\sum_l n_{kl} [\tilde{\cE}_{\cO_k}(\tau_{kl})].
$$
We conclude that $\cE$ is in the image of ${i_{12}}_*: \textup{K}^H(Z_1\cap Z_2)\to \textup{K}^H(X)$.
\end{proof}

It follows from this corollary that we may make the following definition.

\begin{Def} The support of $\cE\in \textup{K}^H(X)$ is the smallest $Z\subset X$ so that $\cE$ is in the  image of $i_*: \textup{K}^H(Z)\to \textup{K}^H(X)$.
\end{Def}

When we write $\cE=\sum n_{kl}[\tilde{\cE}_{\cO_k}(\tau_{kl})]$, the support is the union of the closures of the maximal $\cO_k$'s that appear in the expression with nonzero coefficient $n_{kl}$ for some $l$ (and this is independent of the extensions $\tilde{\cE}_{\cO_k}(\tau_{kl})$).   In fact more is true.

\begin{Cor}\label{cor:lt}
Suppose two basis $\cB^{(1)}:=\{[\widetilde{\cE}_{\cO_k}^{\scriptscriptstyle{(1)}}(\tau_{kl})]\}$ and $\cB^{(2)}:=\{[\widetilde{\cE}_{\cO_k}^{\scriptscriptstyle{(2)}}(\tau_{kl})]\}$ have been chosen as in Theorem \ref{thm:basis}.   Let  $\cE\in \textup{K}^H(X)$  and write 
$$\cE=\sum n_{kl}^{\scriptscriptstyle{(1)}}[\widetilde{\cE}_{\cO_k}^{\scriptscriptstyle{(1)}}(\tau_{kl})]=\sum n_{kl}^{\scriptscriptstyle{(2)}}[\widetilde{\cE}_{\cO_k}^{\scriptscriptstyle{(2)}}(\tau_{kl})].$$
Then for each maximal $\cO_q$, we have $n_{ql}^{(1)}=n_{ql}^{(2)}$, for all $l$.
\end{Cor}

\begin{proof}  We proceed by induction on the number of orbits of $H$ in $X$.  Suppose $\cO_{q}$ is maximal in one (hence the other) expression for $\cE$.  \\
\underline{Case 1.}  $\cO_{q}$ is open in $X$.  Let $Z=X\smallsetminus \cO_{q}$ and $i:Z\hookrightarrow X$ and $j:\cO_{q}\hookrightarrow X$.  Then since $j^*([\widetilde{\cE}_{\cO_q}^{\scriptscriptstyle{(p)}}(\tau_{ql})])=[\cE_{\cO_q}(\tau_{ql})]\in \textup{K}^H(\cO_{q})$, for $p=1,2$,  and  $j^*([\widetilde{\cE}_{\cO_k}^{\scriptscriptstyle{(p)}}(\tau_{kl})])=0, k\neq q$, we have
\begin{align*}
j^*(\cE)&=\sum_{k,l} n_{kl}^{\scriptscriptstyle{(1)}} j^*([\cE_{\cO_k}(\tau_{kl})])=\sum_{l} n_{ql}^{\scriptscriptstyle{(1)}} [\cE_{\cO_q}(\tau_{ql})] \text{ and }  \\
j^*(\cE)&=\sum_{k,l} n_{kl}^{\scriptscriptstyle{(2)}} j^*([\cE_{\cO_k}(\tau_{kl})]) =\sum_{l} n_{ql}^{\scriptscriptstyle{(2)}} [\cE_{\cO_q}(\tau_{ql})],
\end{align*}
in $ \textup{K}^H(\cO_q)$.
By (\ref{thm:basis-a}), $n_{ql}^{\scriptscriptstyle{(1)}}=n_{ql}^{\scriptscriptstyle{(2)}}$, for all $l$. \\
\underline{Case 2.}   $\cO_{q}$ is not open in $X$.  Then $\supp(\cE)\subsetneq X$.  By induction (and the injectivity of $i_*: \textup{K}^H(\supp(\cE))\to \textup{K}^H(X)$) the statement follows.
\end{proof}


\begin{thebibliography}{10}

\bibitem{Achar01}
P.~Achar, \emph{Equivariant coherent sheaves on the nilpotent cone for complex
  reductive lie groups}, Ph.D. thesis, MIT, 2001.

\bibitem{Alvis05}
D.~Alvis, \emph{Induce/restrict matrices for exceptional {W}eyl groups},
  arXiv:math/0506377v1 (2005).

\bibitem{Carter85}
R.~W. Carter, \emph{Finite {G}roups of {L}ie {T}ype: {C}onjugacy {C}lasses and
  {C}omplex characters}, Pure and Applied Mathematics, John Wiley \& Sons,
  Inc., New York, 1985.

\bibitem{ChrissGinzburg97}
N.~Chriss and V.~Ginzburg, \emph{Representation {T}heory and {C}omplex
  {G}eometry}, Birk\"auser, Boston, MA, 1997.

\bibitem{CollingwoodMcGovern93}
D.~Collingwood and W.~McGovern, \emph{Nilpotent {O}rbits in {S}emisimple {L}ie
  {A}lgebras}, Van Nostrand Reinhold Co., New York, 1993.

\bibitem{Hartshorne77}
R.~Hartshorne, \emph{Algebraic {G}eometry}, Grad. Texts in Math., vol.~52,
  Springer Verlag, 1977.

\bibitem{HuangPandzic02}
J.-S. Huang and P.~Pand\v zi\'c, \emph{Dirac cohomology, unitary
  representations and a proof of a conjecture of {V}ogan}, J. Amer. Math. Soc.
  \textbf{15} (2002), no.~1, 185--202.

\bibitem{KnappVogan95}
A.~W. Knapp and D.~A. Vogan, \emph{Cohomological {I}nduction and {U}nitary
  {R}epresentations}, Princeton Mathematical Series, vol.~45, Princeton
  University Press, 1995.

\bibitem{KostantRallis71}
B.~Kostant and S.~Rallis, \emph{Orbits and representations associated with
  symmetric spaces}, Am.~J.~Math. \textbf{85} (1971), 753--809.

\bibitem{Lusztig79}
G.~Lusztig, \emph{A class of irreducible representations of a {W}eyl group},
  Proc. Nederl. Akad. \textbf{422} (1979), 323--335.

\bibitem{Macdonald72}
I.~G. Macdonald, \emph{Some irreducible representations of {W}eyl groups},
  Bull. London Math. Soc. \textbf{4} (1972), 148--150.

\bibitem{MehdiPandzicVogan15}
S.~Mehdi, P.~Pand\v zi\'c, and D.~Vogan, \emph{Translation principle for
  {D}irac index}, Amer.~J.~Math. \textbf{139} (2017), no.~6, 1465--1491.

\bibitem{MehdiPandzicVoganZierau17b}
S.~Mehdi, P.~Pand\v zi\'c, D.~Vogan, and R.~Zierau, \emph{Computing the
  associated cycles of certain {H}arish-{C}handra modules}, preprint, 2017.

\bibitem{Parthasarathy72}
R.~Parthasarathy, \emph{Dirac operator and discrete series}, Ann. of Math.
  \textbf{96} (1972), 1--30.

\bibitem{Quillen73}
D.~Quillen, \emph{Higher algebraic {$K$}-theory. {I}}, Algebraic {$K$}-theory,
  {I}: {H}igher {$K$}-theories ({P}roc. {C}onf., {B}attelle {M}emorial {I}nst.,
  {S}eattle, {W}ash., 1972), Springer, Berlin, 1973, pp.~85--147. Lecture Notes
  in Math., Vol. 341.

\bibitem{SchmidVilonen00}
W.~Schmid and K.~Vilonen, \emph{Characteristic cycles and wave front cycles of
  representations of reductive {L}ie groups}, Ann. of Math. (2) \textbf{151}
  (2000), no.~3, 1071--1118.

\bibitem{Springer78}
T.~A. Springer, \emph{A construction of representations of {W}eyl groups},
  Invent.~Math. \textbf{44} (1978), 279--293.

\bibitem{Steinberg76}
R.~Steinberg, \emph{On the desingularization of the unipotent variety},
  Invent.~Math. \textbf{36} (1976), 209--244.

\bibitem{Thomason87}
R.~W. Thomason, \emph{Algebraic {$K$}-theory of group scheme actions},
  Algebraic topology and algebraic {$K$}-theory, Ann. of Math. Stud., vol. 113,
  Princeton Univ. Press, Princeton, NJ, 1987, pp.~539--563.

\bibitem{Vogan81}
D.~A. Vogan, \emph{Representations of {R}eal {R}eductive {L}ie {G}roups},
  Progress in Mathematics, vol.~15, Birkh\"auser, Boston, 1981.

\bibitem{Vogan91}
D.~A. Vogan, \emph{Associated varieties and unipotent representations},
  Harmonic Analysis on reductive groups (W.~Barker and P.~Sally, eds.),
  Progress in Mathematics, vol. 101, Birkh\"auser, 1991, pp.~315--388.

\bibitem{Vogan98}
D.~A. Vogan, \emph{The method of coadjoint orbits for real reductive groups},
  Representation theory of {L}ie groups ({P}ark {C}ity, {UT}, 1998), IAS/Park
  City Math. Ser., vol.~8, Amer. Math. Soc., Providence, RI, 2000,
  pp.~179--238.

\bibitem{Weibel13}
C.~Weibel, \emph{The {$K$}-book: an introduction to algebraic {$K$}-theory},
  Graduate Studies in Mathematics, vol. 145, Am.~Math.~Soc., Providence, 2013.

\end{thebibliography}

\providecommand{\bysame}{\leavevmode\hbox to3em{\hrulefill}\thinspace}
\providecommand{\MR}{\relax\ifhmode\unskip\space\fi MR }
\providecommand{\MRhref}[2]{%
  \href{http://www.ams.org/mathscinet-getitem?mr=#1}{#2}
}
\providecommand{\href}[2]{#2}

\end{document}